\documentclass[12pt]{amsart}
\usepackage{amsmath,amsthm,amsfonts,amscd,amssymb,eucal,latexsym,mathrsfs, appendix, yhmath}
\usepackage[numbers,sort&compress]{natbib}
\usepackage[all,cmtip]{xy}
\usepackage{color}

\newtheorem{theorem}{Theorem}[section]
\newtheorem{corollary}[theorem]{Corollary}
\newtheorem{lemma}[theorem]{Lemma}
\newtheorem{proposition}[theorem]{Proposition}
\newtheorem{remark}[theorem]{Remark}

\theoremstyle{definition}

\newtheorem{question}[theorem]{Question}

\newcommand{\Zb}{{\mathbb Z}}

\newcommand{\Nb}{{\mathbb N}}

\newcommand{\cB}{{\mathcal B}}
\newcommand{\cU}{{\mathcal U}}

\newcommand{\cV}{{\mathcal V}}

\newcommand{\cW}{{\mathcal W}}

\newcommand{\cE}{{\mathcal E}}
\newcommand{\cA}{{\mathcal A}}
\newcommand{\cD}{{\mathcal D}}

\newcommand{\fF}{{\mathfrak F}}

\newcommand{\sym}{{\rm Sym}}
\newcommand{\map}{{\rm Map}}

\newcommand{\cC}{{\mathcal C}}
\newcommand{\cF}{{\mathcal F}}

\allowdisplaybreaks
\numberwithin{equation}{section}

\begin{document}
\title[weak expansiveness for actions of sofic groups]{weak expansiveness for actions of sofic groups}
\author[Nhan-Phu Chung and Guohua Zhang]{Nhan-Phu Chung and Guohua Zhang}

\address{\hskip- \parindent
Nhan-Phu Chung, Department of Mathematics, University of Sciences, Vietnam National University at Ho Chi Minh city, Vietnam}
\email{cnphu@hcmus.edu.vn}

\address{\hskip- \parindent
Guohua Zhang, School of Mathematical Sciences and LMNS, Fudan University and Shanghai Center for Mathematical Sciences, Shanghai 200433, China}
\email{chiaths.zhang@gmail.com}

\date{\today}

\begin{abstract}
In this paper, we shall introduce $h$-expansiveness and asymptotical $h$-expansiveness for actions of sofic groups. By definition, each $h$-expansive action of a sofic group is asymptotically $h$-expansive. We show that each expansive action of a sofic group is $h$-expansive, and, for any given asymptotically $h$-expansive action of a sofic group, the entropy function (with respect to measures) is upper semi-continuous and hence the system admits a measure with maximal entropy.

Observe that asymptotically $h$-expansive property was first introduced and studied by Misiurewicz for $\mathbb{Z}$-actions using the language of tail entropy. And thus in the remaining part of the paper, we shall compare our definitions of weak expansiveness for actions of sofic groups with the definitions given in the same spirit of Misiurewicz's ideas when the group is amenable. It turns out that these two definitions are equivalent in this setting.
\end{abstract}





\maketitle


\section{Introduction}

Dynamical system theory is the study of qualitative properties of group actions on spaces
with certain structures. In this paper, by a \emph{topological dynamical system} we mean a continuous action of a countable discrete sofic group on a compact metric space.
Sofic groups were defined implicitly by Gromov in \cite{Gro} and explicitly by Weiss in \cite{Wei}. They include all amenable groups and residually finite groups.

Recently, Lewis Bowen introduced a notion of entropy for measure-preserving actions of a countable discrete sofic group admitting a generating measurable partition with finite entropy \cite{Bowen10, BETDS}. The main idea here is to replace the important F\o lner sequence of a countable discrete amenable group with a sofic approximation for a countable discrete sofic group. Very soon after \cite{Bowen10}, in the spirit of L. Bowen's sofic measure-theoretic entropy, Kerr and Li developed an operator-algebraic approach to sofic entropy \cite{KL11Inven, KLAJM} which applies not only to continuous actions of countable discrete sofic groups on compact metric spaces but also to all measure-preserving actions of countable discrete sofic groups on standard probability measure spaces. From then on, there are many other papers, presenting different but equivalent definitions of sofic entropy \cite{KGGD, Zhang12}, extending sofic entropy to sofic pressure \cite{CETDS} and to sofic mean dimension \cite{LAM}, and discussing combinatorial independence for actions of sofic groups \cite{KLCMS}.

\medskip

Let $X$ be a compact metric space. Any homeomorphism $T: X\rightarrow X$ generates naturally a topological dynamical system by considering the group $\{T^n: n\in \mathbb{Z}\}$. Even the given map $T: X\rightarrow X$ is just a continuous map (may be non-invertible), we still call it a topological dynamical system by considering the semi-group $\{T^n: n\in \mathbb{Z}_+\}$.
A self-homeomorphism of a compact metric space is \emph{expansive} if, for each pair of distinct points, some iterate of the homeomorphism separates them by a definite amount. Expansiveness is in fact a multifaceted dynamical condition which plays a very important role in the exploitation of hyperbolicity in smooth dynamical systems \cite{Mane}.
In the setting of considering a continuous mapping over a compact metric space, two classes of weak expansiveness, the
$h$-expansiveness and asymptotical $h$-expansiveness, were introduced by Rufus Bowen \cite{Bex} and Misiurewicz \cite{Mi1}, respectively.
By definition, each $h$-expansive system is asymptotically $h$-expansive.
Both of $h$-expansiveness and asymptotical $h$-expansiveness turn out to be important in the study of smooth dynamical systems \cite{Burguet, DFPV, DMaass, DownNew, LVY}.

It is direct to define expansiveness for actions of groups. That is, let $G$ be a discrete group acting on a compact metric space $X$ (with the metric $\rho$), then we say that \emph{$(X, G)$ is expansive} if there is $\delta> 0$ such that for any two different points $x_1$ and $x_2$ in $X$ there exists $g\in G$ with $\rho (g x_1, g x_2)> \delta$. Such a $\delta$ is called an \emph{expansive constant}. Symbolic systems are standard examples for expansive actions. This obvious extension of the notion of expansiveness has been investigated extensively in algebraic actions for $\Zb^d$ \cite{ES1,LS,Schmidt90,Schmidt} and for more general groups \cite{Bo11,ChungLi,DS}; and a general framework of dynamics of $d\in \mathbb{N}$ commuting homeomorphisms over a compact metric space, in terms of expansive behavior along lower dimensional subspaces of $\mathbb{R}^d$, was first proposed by Boyle and Lind \cite{Blind}.

Now, a natural question rises: how to define weak expansiveness using entropy techniques when considering actions of countable sofic groups?

\medskip

The problem is addressed in this paper. When considering a continuous mapping over a compact metric space, R. Bowen introduced $h$-expansiveness using separated and spanning subsets by considering topological entropy of special subsets \cite{Bex}, and then Misiurewicz introduced asymptotical $h$-expansiveness, which is weaker than $h$-expansiveness, using open covers by introducing tail entropy \cite{Mi1}. Observe that the notion of tail entropy was first introduced by Misiurewicz in \cite{Mi1} where he called it topological conditional entropy, and then Buzzi called this quantity local entropy in \cite{Bu}. Here, we follow Downarowicz and Serafin \cite{DowSer} and Downarowicz \cite{Dow11}. It was Li who first used open covers for actions of sofic groups to consider sofic mean dimension \cite{LAM}, and then this idea was used in \cite{Zhang12} to consider equivalently the entropy for actions of sofic groups. To fix the problem of defining weak expansiveness naturally for actions of sofic groups, we shall use open covers again to introduce the properties of $h$-expansiveness and asymptotical $h$-expansiveness in the spirit of Misiurewicz \cite{Mi1}. The idea turns out to be successful. From definition each $h$-expansive action of a sofic group is asymptotically $h$-expansive;
we shall prove that each expansive action of a sofic group is $h$-expansive (Theorem \ref{1207121257}), and hence, $h$-expansiveness and asymptotical $h$-expansiveness are indeed two classes of weak expansiveness. Additionally, similar to the setting of considering a continuous mapping over a compact metric space,
 for any given asymptotically $h$-expansive action of a sofic group, the entropy function (with respect to measures) is upper semi-continuous (Theorem \ref{1207121342}) and hence the system admits a measure with maximal entropy.

 Observe that the asymptotically $h$-expansive property was first introduced and studied by Misiurewicz for $\mathbb{Z}$-actions using the language of tail entropy. We can define tail entropy for actions of amenable groups in the same spirit, and so it is quite natural to ask if we could define asymptotical $h$-expansiveness for actions of amenable groups along the line of tail entropy. The answer turns out to be true, that is, our definitions of weak expansiveness for actions of sofic groups are equivalent to the definitions given in the same spirit of Misiurewicz's ideas of using tail entropy when the group is amenable (Theorem \ref{1306160126}).
In \cite{Mi1} Misiurewicz provided a typical example of an asymptotically $h$-expansive system, that is, any continuous endomorphism of a compact metric group with finite entropy is asymptotically $h$-expansive. We shall show that in fact this holds in a more general setting with the help of Theorem \ref{1306160126}, precisely, any action of a countable discrete amenable group acting on a compact metric group by continuous automorphisms is asymptotically $h$-expansive if and only if the action has finite entropy (Theorem \ref{1307122130}).

\medskip

The paper is organized as follows. In section 2 we prove that each expansive action of a sofic group admits a measure with maximal entropy based on the sofic measure-theoretic entropy introduced in \cite{Bowen10}. In section 3 we introduce $h$-expansive and asymptotically $h$-expansive actions of sofic groups in the spirit of Misiurewicz \cite{Mi1}. Each $h$-expansive action of a sofic group is asymptotically $h$-expansive by the definitions. We show that each expansive action of a sofic group is $h$-expansive, and each asymptotically $h$-expansive action of a sofic group admits a measure with maximal entropy (in fact, its entropy function is upper semi-continuous with respect to measures). In section 4 we present our first interesting non-trivial example of an $h$-expansive action of a sofic group which is in fact the profinite action of a countable group.
In order to understand further our introduced weak expansiveness for actions of sofic groups in the setting of amenable groups, in section 5 we define tail entropy for actions of amenable groups in the same spirit of Misiurewicz. And then in section 6 we compare our definitions of weak expansiveness for actions of sofic groups with the definitions given in section 5 when the group is amenable. It turns out that these two definitions are equivalent in this setting. And then in section 7 we show that any action of a countable discrete amenable group acting on a compact metric group by continuous automorphisms is asymptotically $h$-expansive if and only if the action has finite entropy.

\section{Expansive actions of sofic groups}

 Let $G$ be a countable discrete group.
For each $d\in\Nb$, denote by $\sym(d)$ the permutation group of $\{1, \cdots, d\}$. We say that \emph{$G$ is sofic} if there is a sequence $\Sigma=\{\sigma_i: G \rightarrow \sym(d_i), g\mapsto \sigma_{i, g}, d_i\in \Nb\}_{i\in \Nb}$ such that
$$\lim_{i\rightarrow \infty}\frac{1}{d_i} |\{a\in \{1, \cdots, d_i\}: \sigma_{i,s}\sigma_{i,t}(a)=\sigma_{i,st}(a)\}|=1\ \text{for all}\ s, t\in G$$
and
$$\lim_{i\rightarrow \infty}\frac{1}{d_i} |\{a\in \{1, \cdots, d_i\}: \sigma_{i,s}(a)\neq \sigma_{i,t}(a)\}|=1\ \text{for all distinct}\ s, t\in G.$$
Here, by $|\bullet|$ we mean the cardinality of a set $\bullet$. Such a sequence $\Sigma$ with $\lim_{i\rightarrow \infty} d_i= \infty$ is referred as a \emph{sofic approximation of $G$}. Observe that the condition $\lim_{i\rightarrow \infty} d_i= \infty$ is essential for the variational principle concerning entropy of actions of sofic groups (see \cite{KL11Inven} and \cite{Zhang12} for the global and local variational principles, respectively), and it is automatic if $G$ is infinite.

\medskip

{\it Throughout the paper, $G$ will be a countable discrete sofic group, with a fixed sofic approximation $\Sigma$ as above and $G$ acts on a compact metric space $(X, \rho)$.
}

\medskip

In this section, based on the sofic measure-theoretic entropy introduced in \cite{Bowen10}, we mainly prove that, for an expansive action of a sofic group, the entropy function is upper semi-continuous with respect to measures, and hence the action admits a measure with maximal entropy. Additionally, we show that in general the entropy function of a finite open cover is also upper semi-continuous with respect to measures.

Denote by $M (X)$ the set of all Borel probability measures on $X$, which is a compact metric space if endowed with the well-known weak star topology; and by $M (X, G)$ the set of all $G$-invariant elements $\mu$ in $M (X)$, i.e., $\mu (A)= \mu (g^{- 1} A)$ for each $g\in G$ and all $A\in \mathcal{B}_X$, where $\mathcal{B}_X$ is the Borel $\sigma$-algebra of $X$. Note that if $M (X, G)\neq \emptyset$ then it is a compact metric space.

\medskip

For a set $Y$, we denote by $\fF_Y$ the set of all non-empty finite subsets of $Y$. By a \emph{cover} of $X$ we mean a family of subsets of $X$ with the whole space as its union. If elements of a cover are pairwise disjoint, then it is called a \emph{partition}.
Denote by $\mathcal{C}_X, \mathcal{C}_X^o, \mathcal{C}_X^c$ and $\mathcal{P}_X$
the set of all finite Borel covers, finite open covers, finite closed covers and finite Borel partitions of $X$, respectively. For $\mathcal{V}\in \mathcal{C}_X$ and $\emptyset\neq K\subset X$ we set $\overline{\mathcal{V}}= \{\overline{V}: V\in \mathcal{V}\}$, and set $N (\mathcal{V}, K)$ to be the minimal cardinality of sub-families of $\mathcal{V}$ covering $K$ (with $N (\mathcal{V}, \emptyset)= 0$ by convention).

\medskip

\emph{Throughout the whole paper, we shall fix the convention $\log 0 = - \infty$.}

\medskip

Now we recall the sofic measure-theoretic entropy introduced in \cite[\S 2]{Bowen10}.

Let $\alpha= \{A_1, \cdots, A_k\}\in \mathcal{P}_X, k\in \mathbb{N}$ and $\sigma: G\rightarrow \sym(d), d \in \mathbb{N}$, and let $\zeta$ be the uniform probability measure on $\{1, \cdots, d\}$ and $\beta= \{B_1, \cdots, B_k\}$ a partition of $\{1, \cdots, d\}$. Assume $\mu\in M (X, G)$.
 For $F\in\mathfrak{F}_G$, we denote by $\map(F, k)$ the set of all functions $\phi: F\rightarrow \Nb$ such that $\phi(f)\leq k$ for all $f\in F$, and we set
  $$d_F(\alpha,\beta)=\sum_{\phi\in \map(F,k)}|\mu(A_\phi)-\zeta(B_\phi)|,$$
where
$$A_\phi=\bigcap_{f\in F}f^{-1}A_{\phi(f)}\ \text{and}\ B_\phi=\bigcap_{f\in F} \sigma(f)^{-1}B_{\phi(f)}\ \text{for each $\phi\in\map(F,k)$}.$$
Now for each $\varepsilon>0$, let $AP_\mu (\sigma,\alpha:F,\varepsilon)$ (or just $AP(\sigma,\alpha:F,\varepsilon)$ if there is no any ambiguity) be the set of all partitions $\beta=\{B_1, \cdots, B_k\}$ of $\{1, \cdots, d\}$ with $d_F(\alpha,\beta)\leq \varepsilon$. In particular, $|AP(\sigma,\alpha:F,\varepsilon)|\le k^d$.
We define
$$H_{\mu,\Sigma}(\alpha:F,\varepsilon)=\limsup_{i\rightarrow \infty}\frac{1}{d_i}\log|AP(\sigma_i,\alpha:F,\varepsilon)|\le \log |\alpha|,$$
$$H_{\mu,\Sigma}(\alpha:F)=\lim_{\varepsilon\rightarrow 0}H_{\mu,\Sigma}(\alpha:F,\varepsilon)= \inf_{\varepsilon> 0} H_{\mu,\Sigma}(\alpha:F,\varepsilon) \le \log |\alpha|,$$
$$h_{\mu,\Sigma}(\alpha)=\inf_{F\in\fF_G}H_{\mu,\Sigma}(\alpha:F)\le \log |\alpha|.$$
Observe that $AP(\sigma_i,\alpha:F,\varepsilon)$ may be empty, and in the case that $AP(\sigma_i,\alpha:F,\varepsilon)= \emptyset$ for all large enough $i\in \mathbb{N}$ we have $H_{\mu,\Sigma}(\alpha:F, \varepsilon)= -\infty$ by the convention $\log 0 = - \infty$. Hence, $h_{\mu,\Sigma}(\alpha)$ may take a value in $[0, \log |\alpha|]\cup \{- \infty\}$.

The main result of \cite{Bowen10} tells us that, if there exists an $\alpha\in \mathcal{P}_X$ generating the $\sigma$-algebra $\mathcal{B}_X$ (in the sense of $\mu$) then the quantity $h_{\mu,\Sigma} (\alpha)$ is independent of the selection of such a partition, and this quantity, denoted by $h_{\mu,\Sigma}(X,G)$, is called the \emph{measure-theoretic $\mu$-entropy of $(X, G)$}.
Indeed, L. Bowen defined the measure-theoretic entropy in a more general case when the action
admits a generating partition $\beta$ (not necessary finite) with finite Shannon entropy \cite{Bowen10}. We say that the partition $\beta\subset \mathcal{B}_X$ \emph{generates the $\sigma$-algebra $\mathcal{B}_X$ (in the sense of $\mu$)} if for each $B\in \mathcal{B}_X$ there exists $A\in \mathcal{A}$ such that $\mu (A\Delta B)= 0$, where $\mathcal{A}$ is the smallest $G$-invariant sub-$\sigma$-algebra of $\mathcal{B}_X$ containing $\beta$.

 Observe that, for an expansive action $(X, G)$ of a sofic group with an expansive constant $\delta> 0$, if $\xi\in \mathcal{P}_X$ satisfies $\text{diam} (\xi)< \delta$, where $\text{diam} (\xi)$ denotes the maximal diameter of subsets in $\xi$, then, for each $\mu\in M (X, G)$ (if $M (X, G)\neq \emptyset$), $\xi$ generates $\mathcal{B}_X$ \cite[Theorem 5.25]{Walters82}, and so the quantity $h_{\mu,\Sigma}(X,G)$ is well defined.

 \medskip

{\it For technical reasons for $r_1, r_2\in [- \infty, \infty]$ we set $r_1+ r_2= - \infty$ by convention in the case that either $r_1= - \infty$ or $r_2= - \infty$, and for $r_1, r_2\in (- \infty, \infty]$ we set $r_1+ r_2= \infty$ by convention in the case that either $r_1= \infty$ or $r_2= \infty$.
}

\medskip

We say that a function $f: Y\rightarrow [- \infty, \infty)$ defined over a compact metric space $Y$ is \emph{upper semi-continuous} if $\limsup_{y'\rightarrow y} f (y')\le f (y)$ for each $y\in Y$.
 The following result shows that
each expansive action of a sofic group admits a measure with maximal entropy.

\begin{theorem}\label{expansivesofic}
Let $(X, G)$ be an expansive action of a sofic group with $M (X, G)\neq \emptyset$. Then $h_{\bullet,\Sigma}(X,G): M(X, G)\rightarrow [0,\infty)\cup \{-\infty\}$ is an upper semi-continuous function.
\end{theorem}
\begin{proof}
The proof is inspired by \cite[Theorem 8.2]{Walters82}.

Let $\delta> 0$ be an expansive constant for $(X, G)$ and
 $\xi\in \mathcal{P}_X$ with $\text{diam} (\xi)< \delta$. Then $\xi$ generates $\mathcal{B}_X$ and so $h_{\mu,\Sigma}(X,G)\in [0, \log |\xi|]\cup \{- \infty\}$ for each $\mu\in M (X, G)$.

Now fix $\eta>0$ and $\mu\in M (X, G)$. It suffices to find an open set $U\subset M (X, G)$ containing $\mu$ such that $h_{\nu,\Sigma}(X,G)\le h_{\mu,\Sigma}(X,G)+\eta$ for each $\nu\in U$.

 We choose $F\in \fF_G$ and $\varepsilon>0$ such that $H_{\mu,\Sigma}(\xi:F,2\varepsilon)\le h_{\mu, \Sigma} (X, G)+ \eta$.
 Say $\xi=\{A_1, \cdots, A_k\}$ and let $0<\varepsilon_1< \frac{\varepsilon}{2M^2}$ with $M=|\map(F,k)|= k^{|F|}$. Let $\phi\in\map(F,k)$. Since $\mu$ is regular, there exists a compact set $K_\phi\subset A_\phi$ with $\mu(A_\phi\setminus K_\phi)<\varepsilon_1$, and then for each $i= 1, \cdots, k$ we define
 $$L_i=\bigcup_{f\in F}\{fK_\phi: \phi(f)=i\}\subset A_i.$$
Then $L_1, \cdots, L_k$ are pairwise disjoint compact subsets of $X$, and so there exists $\xi'= \{A'_1, \cdots, A'_k\}\in \mathcal{P}_X$ such that $\text{diam} (\xi')< \delta$ and, for each $j= 1, \cdots, k$, $L_j\subset \mbox{int} (A'_j)$ where $\mbox{int} (A'_j)$ denotes the interior of $A_j'$. Observe that
$$K_\phi\subset \mbox{int}\left(\bigcap_{f\in F}f^{-1}A'_{\phi(f)}\right)=\mbox{int}(A'_\phi)\ \left(\text{here}\ A'_\phi= \bigcap_{f\in F}f^{-1}A'_{\phi(f)}\right)$$
by the construction of $\xi'$, and so using Urysohn's Lemma we can choose $u_\phi\in C(X)$, where $C (X)$ denotes the set of all real-valued continuous functions over $X$, with $0\leq u_\phi\leq 1$ which equals 1 on $K_\phi$ and vanishes on $X\setminus\mbox{int}(A'_\phi)$.
Set
$$U=\{\nu\in M(X, G): |\nu(u_\phi)-\mu(u_\phi)|<\varepsilon_1 \mbox{ for all } \phi\in \map(F,k)\}$$
which is an open set of $M (X, G)$ containing $\mu$. Let $\nu\in U$. Then
$\nu(A'_\phi)\geq \nu(u_\phi)>\mu(u_\phi)-\varepsilon_1\geq \mu(K_\phi)-\varepsilon_1$
and hence $\mu(A_\phi)-\nu(A'_\phi)<2\varepsilon_1$ for each $\phi\in\map(F,k)$. Observe $\{A_\phi: \phi\in \map(F,k)\}\in \mathcal{P}_X$ and $\{A'_\phi: \phi\in \map(F,k)\}\in \mathcal{P}_X$. Note that if $p_1, \cdots, p_m, q_1, \cdots, q_m, c$ are nonnegative real numbers with $m\in \mathbb{N}$ such that $\sum_{i= 1}^m p_i= \sum_{i= 1}^m q_i=1$ and $p_j-q_j<c$ for each $j= 1, \cdots, m$ then
$$q_i-p_i= \sum_{j\neq i} (p_j-q_j)<mc$$
 and hence $|p_i-q_i|<mc$ for any $i= 1, \cdots, m$. This implies
$\big|\nu(A'_\phi)-\mu(A_\phi)\big|<2\varepsilon_1 M$
for each $\phi\in\map(F,k)$, and so
$$\sum_{\phi\in\map(F,k)}|\nu(A'_\phi)-\mu(A_\phi)|\leq 2\varepsilon_1 M^2\leq\varepsilon.$$
 Thus $AP_{\nu}(\sigma_i,\xi':F,\varepsilon)\subset AP_{\mu}(\sigma_i,\xi:F,2\varepsilon)$ for each $i\in \Nb$, and then $$H_{\nu,\Sigma}(\xi':F)\leq H_{\nu,\Sigma}(\xi':F,\varepsilon)\leq H_{\mu,\Sigma}(\xi:F,2\varepsilon)\le h_{\mu,\Sigma}(X,G)+\eta.$$
As $\text{diam} (\xi')< \delta$, $\xi'\in \mathcal{P}_X$ generates $\mathcal{B}_X$ by the construction of $\delta$, and so we get $h_{\nu,\Sigma}(X,G)\leq h_{\mu,\Sigma}(X,G)+\eta$ for each $\nu\in U$ as desired. This finishes the proof.
\end{proof}

In the spirit of L. Bowen's entropy as above, Kerr and Li introduced alternatively the sofic measure-theoretic entropy \cite{KL11Inven, KLAJM} as follows.

Let $(Y, \rho)$ be a metric space and $\varepsilon> 0$. A set $\emptyset\neq A\subset Y$ is said to be \emph{$(\rho, \varepsilon)$-separated} if $\rho (x, y)\ge \varepsilon$ for all distinct $x, y\in A$. We write $N_\varepsilon (Y, \rho)$ for the maximal cardinality of finite non-empty $(\rho, \varepsilon)$-separated subsets of $Y$ (and set $N_\varepsilon (\emptyset, \rho)= 0$ by convention). A basic fact is that if $\emptyset\neq A\subset Y$ is a maximal finite $(\rho, \varepsilon)$-separated subset of $Y$ then for each $y\in Y$ there exists $x\in A$ such that $\rho (x, y)< \varepsilon$.

For each $d\in \mathbb{N}$ and $(x_1, \cdots, x_d), (x_1', \cdots, x_d')\in X^d$, we set
\begin{equation} \label{metric}
\rho_d ((x_1, \cdots, x_d), (x_1', \cdots, x_d'))= \max_{i= 1}^d \rho (x_i, x_i').
\end{equation}

For $F\in \mathfrak{F}_G, \delta> 0$ and $\sigma: G\rightarrow Sym (d), g\mapsto \sigma_g$ with $d\in \mathbb{N}$, put
$$X^d_{F, \delta, \sigma}= \left\{(x_1, \cdots, x_d)\in X^d: \max_{s\in F} \sqrt{\sum_{i= 1}^d \frac{1}{d} \rho^2 (s x_i, x_{\sigma_s (i)})}< \delta\right\};$$
and then for $\mu\in M (X)$ and $L\in \mathfrak{F}_{C (X)}$, we set
$$X^d_{F, \delta, \sigma, \mu, L}= \left\{(x_1, \cdots, x_d)\in X^d_{F, \delta, \sigma}: \max_{f\in L} \left|\frac{1}{d} \sum_{i= 1}^d f (x_i)- \mu (f)\right|< \delta\right\}.$$
By \cite[Proposition 3.4]{KLAJM} the \emph{measure-theoretic $\mu$-entropy of $(X, G)$} can be defined as (recalling the convention $\log 0 = - \infty$)
$$h_\mu (G, X)= \sup_{\varepsilon> 0} \inf_{L\in \mathfrak{F}_{C (X)}} \inf_{F\in \mathfrak{F}_G} \inf_{\delta> 0} \limsup_{i\rightarrow \infty} \frac{1}{d_i} \log N_\varepsilon \left(X^{d_i}_{F, \delta, \sigma_i, \mu, L}, \rho_{d_i}\right).$$

The sofic measure-theoretic entropy can be defined equivalently using finite open covers as follows \cite[\S 2]{Zhang12}. We remark that it was Li who first used open covers for actions of sofic groups to consider sofic mean dimension in \cite{LAM}.

For $\mathcal{U}\in \mathcal{C}_X$ and $d\in\Nb$, we denote by $\mathcal{U}^d$ the finite Borel cover of $X^d$ consisting of $U_1\times U_2\times\cdots \times U_d$, where $U_1,\cdots ,U_d\in \cU$. Let $\mathcal{U}\in \mathcal{C}_X$, we set (recalling the convention $\log 0 = - \infty$)
$$h_{F, \delta, \mu, L} (G, \mathcal{U})= \limsup_{i\rightarrow \infty} \frac{1}{d_i} \log N \left(\mathcal{U}^{d_i}, X^{d_i}_{F, \delta, \sigma_i, \mu, L}\right).$$
In particular, $h_{F, \delta, \mu, L} (G, \mathcal{U})$ takes the value $- \infty$ if $X^{d_i}_{F, \delta, \sigma_i, \mu, L}= \emptyset$ for all $i\in \mathbb{N}$ large enough.
Now we define the \emph{measure-theoretic $\mu$-entropy of $\mathcal{U}$} as
$$h_\mu (G, \mathcal{U})= \inf_{L\in \mathfrak{F}_{C (X)}} \inf_{F\in \mathfrak{F}_G} \inf_{\delta> 0} h_{F, \delta, \mu, L} (G, \mathcal{U})\le \log N (\mathcal{U}, X).$$

It is not hard to check that
$$h_\mu (G, X)= \sup_{\mathcal{U}\in \mathcal{C}_X^o} h_\mu (G, \mathcal{U}).$$
Moreover, by the proof of \cite[Theorem 6.1]{KL11Inven}, it was proved implicitly $h_\mu (G, X)= - \infty$ (and hence $h_\mu (G, \mathcal{U})= - \infty$ for all $\mathcal{U}\in \mathcal{C}_X^o$) for each $\mu\in M (X)\setminus M (X, G)$.

Observe that both of $h_\mu (G, \mathcal{U})$ and $h_\mu (G, X)$ may take the value of $- \infty$, and by \cite{KL11Inven, KLAJM} if $\mu\in M (X, G)$ and $\mathcal{B}_X$ admits a generating partition (in the sense of $\mu$) with finite Shannon entropy then $h_\mu (G, X)$ is just the quantity $h_{\mu,\Sigma}(X,G)$ introduced before.

\medskip

The following result is easy to obtain:

\begin{proposition} \label{1207102137}
Let $\mathcal{U}\in \mathcal{C}_X^o$. Then $h_{\bullet} (G, \mathcal{U}): M (X)\rightarrow [0, \log N (\mathcal{U}, X)]\cup \{- \infty\}$
is an upper semi-continuous function.
\end{proposition}
\begin{proof}
Let $\mu\in M (X)$. For any $\varepsilon> 0$ we may choose $L\in \mathfrak{F}_{C (X)}, F\in \mathfrak{F}_G$ and $\delta> 0$ such that $h_{F, 2 \delta, \mu, L} (G, \mathcal{U})\le h_\mu (G, \mathcal{U})+ \varepsilon$. Now we consider the non-empty open set
$$\mu\in V= \{\nu\in M (X): |\nu (f)- \mu (f)|< \delta\ \text{for all}\ f \in L\}.$$
Then for each $\nu\in V$ we have $X^d_{F, \delta, \sigma, \nu, L}\subset X^d_{F, 2 \delta, \sigma, \mu, L}$ for each $\sigma: G\rightarrow Sym (d), d\in \mathbb{N}$, which implies $h_\nu (G, \mathcal{U})\le h_{F, \delta, \nu, L} (G, \mathcal{U})\le h_{F, 2 \delta, \mu, L} (G, \mathcal{U})\le h_\mu (G, \mathcal{U})+ \varepsilon$.
This implies that the considered function is upper semi-continuous.
\end{proof}

\section{Weak expansiveness for actions of sofic groups}

Note that, when considering a continuous mapping over a compact metric space, since the introduction of $h$-expansiveness and asymptotical $h$-expansiveness, both of them turn out to be very important classes in the research area of dynamical systems. It is shown by R. Bowen \cite{Bex} that positively expansive systems,
expansive homeomorphisms, endomorphisms of a compact Lie group and
Axiom $A$ diffeomorphisms are all $h$-expansive, by Misiurewicz
\cite{Mi1} that every continuous endomorphism of a compact metric
group is asymptotically $h$-expansive if it has finite entropy, and
by Buzzi \cite{Bu} that any $C^\infty$ diffeomorphism on a compact
manifold is asymptotically $h$-expansive. Moreover, there are more nice characterizations of asymptotical $h$-expansiveness obtained recently for this setting. For example, a topological dynamical system is asymptotically $h$-expansive if and only if it admits a principal extension to a symbolic system by Boyle and Downarowicz \cite{BD}, i.e., a symbolic extension which preserves entropy for each
invariant measure; if and only if it is hereditarily uniformly
lowerable by Huang, Ye and the second author of the present paper \cite{HYZ} (for a detailed definition of the hereditarily uniformly lowerable property and its story see \cite{HYZ}).

In this section we explore similar weak expansiveness for actions of sofic groups.

\medskip

By \cite[Proposition 2.4]{KLAJM} the \emph{topological entropy of $(X, G)$} can be defined as
$$h (G, X)= \sup_{\varepsilon> 0} \inf_{F\in \mathfrak{F}_G} \inf_{\delta> 0} \limsup_{i\rightarrow \infty} \frac{1}{d_i} \log N_\varepsilon \left(X^{d_i}_{F, \delta, \sigma_i}, \rho_{d_i}\right),$$
which is introduced and discussed in \cite{KL11Inven, KLAJM}.
Before proceeding, we need to recall the topological entropy for actions of sofic groups introduced in \cite[\S 2]{Zhang12} using finite open covers.
Let $\mathcal{U}\in \mathcal{C}_X$.
For $F\in \mathfrak{F}_G$ and $\delta> 0$ we set
$$h_{F, \delta} (G, \mathcal{U})= \limsup_{i\rightarrow \infty} \frac{1}{d_i} \log N \left(\mathcal{U}^{d_i}, X^{d_i}_{F, \delta, \sigma_i}\right).$$
Observe again that $h_{F, \delta} (G, \mathcal{U})$ takes the value of $- \infty$ whenever $X^{d_i}_{F, \delta, \sigma_i}= \emptyset$ for all $i\in \mathbb{N}$ large enough.
Now we define the
 \emph{topological entropy of $\mathcal{U}$} as
$$h (G, \mathcal{U})= \inf_{F\in \mathfrak{F}_G} \inf_{\delta> 0} h_{F, \delta} (G, \mathcal{U})\le \log N (\mathcal{U}, X).$$
It is not hard to check that
$$h (G, X)= \sup_{\mathcal{U}\in \mathcal{C}_X^o} h (G, \mathcal{U}).$$
Observe that both of $h (G, \mathcal{U})$ and $h (G, X)$ may take the value of $- \infty$.

 The sofic topological entropy and sofic measure-theoretic entropy are related to each other \cite[Theorem 6.1]{KL11Inven} and \cite[Theorem 4.1]{Zhang12}: for $\mathcal{U}\in \mathcal{C}_X^o$,
 \begin{equation} \label{1310021806}
 h (G, X)= \sup_{\mu\in M (X, G)} h_\mu (G, X)\ \text{and}\ h (G, \mathcal{U})= \max_{\mu\in M (X, G)} h_\mu (G, \mathcal{U}),
 \end{equation}
 where in the right-hand sides as above we set it to $- \infty$ by convention if $M (X, G)= \emptyset$.

In the spirit of Misiurewicz \cite{Mi1}, the above idea can be used to introduce
 $h$-expansiveness and asymptotical $h$-expansiveness for actions of sofic groups.
 Let $\mathcal{U}_1, \mathcal{U}_2\in \mathcal{C}_X$. For $F\in \mathfrak{F}_G$ and $\delta> 0$ we set
$$h_{F, \delta} (G, \mathcal{U}_1| \mathcal{U}_2)= \limsup_{i\rightarrow \infty} \frac{1}{d_i} \log \max_{V\in \mathcal{U}_2^{d_i}} N \left(\mathcal{U}_1^{d_i}, X^{d_i}_{F, \delta, \sigma_i}\cap V\right)\le h_{F, \delta} (G, \mathcal{U}_1),$$
$$h (G, \mathcal{U}_1| \mathcal{U}_2)= \inf_{F\in \mathfrak{F}_G} \inf_{\delta> 0} h_{F, \delta} (G, \mathcal{U}_1| \mathcal{U}_2)\le h (G, \mathcal{U}_1),$$
$$h (G, X| \mathcal{U}_2)= \sup_{\mathcal{U}_1\in \mathcal{C}_X^o} h (G, \mathcal{U}_1| \mathcal{U}_2)\le h (G, X),$$
$$h^* (G, X)= \inf_{\mathcal{U}_2\in \mathcal{C}_X^o} h (G, X| \mathcal{U}_2)\le h (G, X).$$
Then $h (G, X| \{X\})= h (G, X)$ by the definitions.
We say that $(X, G)$ is \emph{$h$-expansive} if $h (G, X| \mathcal{U})\le 0$ for some $\mathcal{U}\in \mathcal{C}_X^o$, and \emph{asymptotically $h$-expansive} if $h^* (G, X)\le 0$.

\medskip

Each $h$-expansive action of a sofic group is asymptotically $h$-expansive by definition.
The next result shows that each expansive action of a sofic group is $h$-expansive, and thus these two kinds of expansiveness are indeed weak expansiveness.

\begin{theorem} \label{1207121257}
Let $(X, G)$ be an expansive action of a sofic group with $\kappa> 0$ an expansive constant and $\mathcal{U}\in \mathcal{C}_X$. Assume that $\text{diam} (\mathcal{U})\le c \kappa$ for some $c< 1$. Then $h (G, X| \mathcal{U})\le 0$.
\end{theorem}
\begin{proof}
Let $\mathcal{V}\in \mathcal{C}_X^o$ and $\varepsilon> 0$. It suffices to prove that $h (G, \mathcal{V}| \mathcal{U})\le \varepsilon$.

Let $\tau> 0$ be a Lebesgue number of $\mathcal{V}$. As $(X, G)$ is an expansive action of a sofic group with $\kappa> 0$ an expansive constant, it is not hard to choose $F\in \mathfrak{F}_G$ such that $\max_{s\in F} \rho (s x, s x')< \kappa$ implies $\rho (x, x')< \frac{\tau}{2}$ (for example see \cite[Chapter 5, \S 5.6]{Walters82}).
Now let $\delta> 0$ be small enough
such that
\begin{equation} \label{1207162007}
|\Theta_d|\cdot |\mathcal{V}|^{\frac{8 \delta^2 |F|}{(1- c)^2 \kappa^2}\cdot d}< e^{\varepsilon d}
\end{equation}
for all $d\in \mathbb{N}$ large enough, where $\Theta_d$ is the set of all subsets $\theta$ of $\{1, \cdots, d\}$ with
$$|\theta|< \frac{8 \delta^2 |F|}{(1- c)^2 \kappa^2}\cdot d.$$

For
any map $\sigma: G\rightarrow Sym (d), g\mapsto \sigma_g$ with $d\in \mathbb{N}$, recall
$$X^d_{F, \delta, \sigma}= \left\{(x_1, \cdots, x_d)\in X^d: \max_{s\in F} \sqrt{\sum_{i= 1}^d \frac{1}{d} \rho^2 (s x_i, x_{\sigma_s (i)})}< \delta\right\},$$
in particular, once $(x_1^*, \cdots, x_d^*)\in X^d_{F, \delta, \sigma}$ then, by direct calculations,
\begin{equation} \label{1207161920}
\left|\left\{i\in \{1, \cdots, d\}: \rho (s x_i^*, x_{\sigma_s (i)}^*)\ge \frac{(1- c) \kappa}{2}\ \text{for some}\ s\in F\right\}\right|< \frac{4 \delta^2 |F|\cdot d}{(1- c)^2 \kappa^2}.
\end{equation}

Now let $V\in \mathcal{U}^d$ and say $(x_1, \cdots, x_d)\in X^d_{F, \delta, \sigma}\cap V$ (if it is not empty).
For any $(x_1', \cdots, x_d')\in X^d_{F, \delta, \sigma}\cap V$, applying \eqref{1207161920} to $(x_1, \cdots, x_d)$ and $(x_1', \cdots, x_d')$ and observing
$\rho (s x_i, s x_i')\le \rho (s x_i, x_{\sigma_s (i)})+ \rho (x_{\sigma_s (i)}, x_{\sigma_s (i)}')+ \rho (s x_i', x_{\sigma_s (i)}')$,
it is easy to see
$$|\{i\in \{1, \cdots, d\}: \rho (s x_i, s x_i')\ge \kappa\ \text{for some}\ s\in F\}| < \frac{8 \delta^2 |F|}{(1- c)^2 \kappa^2}\cdot d.$$
In other words, if we associate each $\theta\in \Theta_d$ with $X^d_{F, \delta, \sigma, V, \theta}$, the set of all $(x_1', \cdots, x_d')$ in $X^d_{F, \delta, \sigma}\cap V$ with
$\theta= |\{i\in \{1, \cdots, d\}: \rho (s x_i, s x_i')\ge \kappa\ \text{for some}\ s\in F\}|$, then
\begin{equation} \label{1107162009}
\bigcup_{\theta\in \Theta_d} X^d_{F, \delta, \sigma, V, \theta}= X^d_{F, \delta, \sigma}\cap V.
\end{equation}
  For each $\theta\in \Theta_d$, if $(x_1", \cdots, x_d"), (x_1^\#, \cdots, x_d^\#)\in X^d_{F, \delta, \sigma, V, \theta}$ then, by the selection of $F$, $\rho (x_i", x_i)< \frac{\tau}{2}, \rho (x_i^\#, x_i)< \frac{\tau}{2}$ and then $\rho (x_i", x_i^\#)< \tau$ for all $i\in \{1, \cdots, d\}\setminus \theta$, thus $X^d_{F, \delta, \sigma, V, \theta}$ can be covered by at most $|\mathcal{V}|^{|\theta|}$ many elements of $\mathcal{V}^d$. This implies
\begin{eqnarray*}
N (\mathcal{V}^d, X^d_{F, \delta, \sigma}\cap V)&\le & \sum_{\theta\in \Theta_d} N (\mathcal{V}^d, X^d_{F, \delta, \sigma, V, \theta})\ (\text{using \eqref{1107162009}}) \\
&\le & |\Theta_d|\cdot |\mathcal{V}|^{\frac{8 \delta^2 |F|}{(1- c)^2 \kappa^2}\cdot d}< e^{\varepsilon d}\ (\text{using \eqref{1207162007}}),
\end{eqnarray*}
and hence $h (G, \mathcal{V}| \mathcal{U})\le \varepsilon$ by the definition, finishing the proof.
\end{proof}

Observing that, for $\mathcal{V}_1, \mathcal{V}_2\in \mathcal{C}_X$ and $K\subset X$,
\begin{equation} \label{1305162351}
N (\mathcal{V}_1, K)\le N (\mathcal{V}_2, K)\cdot \max_{V\in \mathcal{V}_2} N (\mathcal{V}_1, K\cap V),
\end{equation}
it is direct to obtain the following easy while useful observation.

\begin{lemma} \label{1207121301}
Let $\mathcal{U}_1, \mathcal{U}_2\in \mathcal{C}_X$ and $\mu\in M (X)$. Then
$$h_\mu (G, \mathcal{U}_1)\le h_\mu (G, \mathcal{U}_2)+ h (G, \mathcal{U}_1| \mathcal{U}_2)\ \text{and}\ h_\mu (G, X)\le h_\mu (G, \mathcal{U}_2)+ h (G, X| \mathcal{U}_2),$$
$$h (G, \mathcal{U}_1)\le h (G, \mathcal{U}_2)+ h (G, \mathcal{U}_1| \mathcal{U}_2)\ \text{and}\ h (G, X)\le h (G, \mathcal{U}_2)+ h (G, X| \mathcal{U}_2).$$
\end{lemma}
\begin{proof}
Let $\{F_n: n\in \mathbb{N}\}\subset \mathfrak{F}_G$ increase to the whole group $G$ and $\{\delta_n> 0: n\in \mathbb{N}\}$ decrease to 0. By the definitions it is direct to see
$$h_\mu (G, \mathcal{U})= \inf_{L\in \mathfrak{F}_{C (X)}} \lim_{n\rightarrow \infty} h_{F_n, \delta_n, \mu, L} (G, \mathcal{U})$$
for each $\mathcal{U}\in \mathcal{C}_X$ and
$$h (G, \mathcal{U}_1| \mathcal{U}_2)= \lim_{n\rightarrow \infty} h_{F_n, \delta_n} (G, \mathcal{U}_1| \mathcal{U}_2).$$
Now for each $n\in \mathbb{N}$ and $L\in \mathfrak{F}_{C (X)}$, using \eqref{1305162351} it is easy to obtain
$$h_{F_n, \delta_n, \mu, L} (G, \mathcal{U}_1)\le h_{F_n, \delta_n, \mu, L} (G, \mathcal{U}_2)+ h_{F_n, \delta_n} (G, \mathcal{U}_1| \mathcal{U}_2),$$
and then $h_\mu (G, \mathcal{U}_1)\le h_\mu (G, \mathcal{U}_2)+ h (G, \mathcal{U}_1| \mathcal{U}_2)$ by taking limits on both sides, in particular, $h_\mu (G, \mathcal{U}_1)= -\infty$ when $h_\mu (G, \mathcal{U}_2)= -\infty$ (and hence $h_\mu (G, X)= - \infty$ if and only if $h_\mu (G, \mathcal{U}_2)= -\infty$), which implies $h_\mu (G, X)\le h_\mu (G, \mathcal{U}_2)+ h (G, X| \mathcal{U}_2)$ even if $h (G, X| \mathcal{U}_2)= \infty$ (recalling our technical convention of $- \infty+ \infty= - \infty$). The remaining items can be proved similarly.
\end{proof}

As direct corollaries, we have:

\begin{corollary} \label{1207122131}
$h^* (G, X)< \infty$ if and only if $h (G, X)< \infty$, and $h^* (G, X)= - \infty$ if and only if $h (G, X)= - \infty$.
\end{corollary}

\begin{corollary} \label{1207162016}
Assume that $(X, G)$ is $h$-expansive. Then there exists $\mathcal{U}\in \mathcal{C}_X^o$ with
$$h (G, X)= h (G, \mathcal{U})\ \text{and}\ h_\mu (G, X)= h_\mu (G, \mathcal{U})\ \text{for each}\ \mu\in M (X).$$
\end{corollary}

Moreover, we can prove the following result.

\begin{theorem} \label{1207121342}
Let $\mu\in M (X)$. Then
\begin{equation} \label{1310281611}
\limsup_{\nu\rightarrow \mu} h_\nu (G, X)\le h_\mu (G, X)+ h^* (G, X).
\end{equation}
In particular, if $(X, G)$ is asymptotically $h$-expansive then the function $h_\bullet (G, X): M(X)\rightarrow [0,\infty)\cup \{-\infty\}$ is upper semi-continuous.
\end{theorem}
\begin{proof}
Let $\mathcal{U}\in \mathcal{C}_X^o$. By Proposition \ref{1207102137} and Lemma \ref{1207121301} we have
\begin{eqnarray*}
\limsup_{\nu\rightarrow \mu} h_\nu (G, X)&\le & \limsup_{\nu\rightarrow \mu} h_\nu (G, \mathcal{U})+ h (G, X| \mathcal{U}) \\
&\le & h_\mu (G, \mathcal{U})+ h (G, X| \mathcal{U})\le h_\mu (G, X)+ h (G, X| \mathcal{U}).
\end{eqnarray*}
Then \eqref{1310281611} follows directly by taking the infimum over all $\mathcal{U}\in \mathcal{C}_X^o$.

Now we assume that $(X, G)$ is asymptotically $h$-expansive, that is, $h^* (G, X)\le 0$. Then $h (G, X)< \infty$ by Corollary \ref{1207122131}, and hence $h_\eta (G, X)\in [0,\infty)\cup \{-\infty\}$ for each $\eta\in M (X)$ by \eqref{1310021806} (recalling the fact that $h_\eta (G, X)= - \infty$ for each $\eta\in M (X)\setminus M (X, G)$).
 Moreover, using \eqref{1310281611} we obtain
 $$\limsup_{\nu\rightarrow \mu} h_\nu (G, X)\le h_\mu (G, X)+ h^* (G, X)\le h_\mu (G, X)$$
 from our technical convention.
 This finishes the proof.
\end{proof}

\begin{remark}
As a direct corollary of Theorem \ref{1207121342}, we have: if the action $(X, G)$ is asymptotically h-expansive then both of its sofic topological mean dimension and its sofic metric mean dimension with respect to any compatible metric are at most zero \cite[Proposition 4.3 and Theorem 6.1]{LAM}. For the definition of sofic topological mean dimension and sofic metric mean dimension see \cite[\S 2 and \S 4]{LAM}, respectively.
\end{remark}

Combining Theorem \ref{1207121342} with \eqref{1310021806} one has:

\begin{corollary} \label{1310281312}
Each asymptotically $h$-expansive action of a sofic group admits a measure with maximal entropy.
\end{corollary}

In general, the converse does not hold, for example, any $\mathbb{Z}$-action with infinite entropy admits a measure with infinite entropy, whereas, it is not asymptotically $h$-expansive. Furthermore, \cite[Example 6.4]{Mi1} shows us a $\mathbb{Z}$-action with finite entropy such that it is not asymptotically $h$-expansive, while each invariant measure has maximal entropy.
Observe that, as we shall show later, the definitions of weak expansiveness given here for actions of sofic groups are equivalent to definitions given in the same spirit of Misiurewicz's ideas when the group is amenable.

\begin{remark}
Theorem \ref{expansivesofic} is a consequence of Theorem \ref{1207121257} and Theorem \ref{1207121342}.
\end{remark}

\section{Profinite actions}

In this section we provide our first interesting non-trivial $h$-expansive action of a sofic group using the language of the profinite action.

\medskip

Recall that the action $(X, G)$ is \emph{distal} if $\inf_{g\in G} \rho (g x, g y)> 0$ for all distinct
$x, y\in X$, and \emph{equicontinuous} if for each $\delta> 0$ there exists $\varepsilon> 0$ such that $\rho (x, y)\le \varepsilon$ implies $\rho (g x, g y)\le \delta$ for all $g\in G$.

The following result should be known, we provide here a proof for completeness.

 \begin{lemma} \label{easy}
Assume that the action $(X, G)$ is equicontinuous. Then it is distal. And if additionally $X$ is infinite then it is not expansive.
 \end{lemma}
 \begin{proof}
First we prove that $(X, G)$ is distal. Else, there exist distinct points $x_1, x_2\in X$ with $\inf_{g\in G} \rho (g x_1, g x_2)= 0$. As $(X, G)$ is equicontinuous, there exists $\varepsilon> 0$ such that $\rho (x, y)< \varepsilon$ implies $\rho (g x, g y)\le \frac{1}{2} \rho (x_1, x_2)$ for each $g\in G$. Let $g'\in G$ be such that $\rho (g' x_1, g' x_2)< \varepsilon$. Then $0< \rho (x_1, x_2)= \rho ((g')^{- 1} g' x_1, (g')^{- 1} g' x_2)\le \frac{1}{2} \rho (x_1, x_2)$ by the selection of $\varepsilon$, a contradiction.

Now additionally we assume that $X$ is infinite. If $(X, G)$ is expansive, then there exists $\delta> 0$ with  $\sup_{g\in G} \rho (g x, g y)> \delta$ for all distinct
$x, y\in X$. Using again the equicontinuity of $(X, G)$, we could choose $\varepsilon'> 0$ such that $\rho (x, y)< \varepsilon'$ implies $\rho (g x, g y)\le \delta$ for each $g\in G$. As $X$ is infinite, by the compactness of $X$ we could choose distinct points $y_1, y_2\in X$ with $\rho (y_1, y_2)< \varepsilon'$, a contradiction to the selection of $\delta$. That is, $(X, G)$ is not expansive, finishing the proof.
 \end{proof}

  See \cite{Au} for a more detailed story of distal actions and equicontinuous actions.

  \medskip

  Observe that, by definition each action of a sofic group is $h$-expansive if it has topological entropy at most zero, and Kerr and Li proved that each distal action of a sofic group has topological entropy at most zero \cite[Corollary 8.5]{KLCMS}. Note that \cite[Corollary 8.5]{KLCMS} is stated for the sofic topological entropy defined using ultrafilter. However, the limsup version of it follows directly from the ultrafilter version,  since one can pass to a subsequence where the quantity converges and then one can use any free ultrafilter on this subsequence. Thus each distal action of a sofic group is $h$-expansive.
With the help of this observation, we can provide our first interesting $h$-expansive example.

Let $G$ be a countable group. A \textit{chain} of $G$ is a sequence $G= G_0\geq G_1\geq \cdots$ of subgroups with finite indices in $G$. For a chain $(G_n)$, we have a tree structure $T (G, (G_n))$ defined naturally as follows. The vertices are $\{g G_n: n\in \Nb, g\in G\}$ and $(g_1 G_n, g_2 G_m)$ is an edge if $m= n+ 1$ and $g_2 G_m\subset g_1 G_n$. The boundary $\partial T (G, (G_n))$ of $T (G, (G_n))$ consists of all sequences $(x_0, x_1, \cdots)$ of vertices with $x_n$ adjacent to $x_{n+ 1}$ for each $n\in \mathbb{Z}_+$. Then $\partial T (G, (G_n))$ is a compact metrizable space endowed with the topology generated by the open basis consisting of all subsets $O_x= \{(x_0, x_1, \cdots)\in \partial T (G, (G_n)): x_N= x\}$ with $x\in G/G_N$ and $N\in \mathbb{Z}_+$. The natural left actions of $G$ on $G/G_n$ induce the \emph{profinite action} $(\partial T (G, (G_n)), G)$, an action of $G$ on $\partial T (G, (G_n))$ by homeomorphisms.
 In this case the profinite action $(\partial T (G, (G_n)), G)$ is equicontinuous, since for any $x= g G_n\in G/G_n$ and $h\in G$ we have $h O_x= O_y$ with $y= h g G_n\in G/G_n$.

 Combining with the above discussions, we obtain:

 \begin{proposition} \label{1310022211}
 Let $G$ be a countable sofic group and $(G_n)$ a chain of $G$. Then the profinite action of $G$ on $\partial T (G, (G_n))$ is h-expansive. Furthermore if $\partial T (G, (G_n))$ of $T (G, (G_n))$ is infinite then the action is not expansive.
 \end{proposition}

 Dynamical properties of profinite actions have been studied extensively for residually finite groups in \cite{AE,AN}. They were also used to investigate orbit equivalence rigidity for Kazhdan property (T) groups \cite{Ioana, OP}.

\section{Tail entropy for actions of countable discrete amenable groups}

In this section we shall introduce tail entropy for actions of countable discrete amenable groups in the same spirit of Misiurewicz.

\medskip

Recall that $G$ is a countable discrete group. Denote by $e_G$ the unit of $G$.
$G$ is called \emph{amenable}, if there exists a sequence $\{F_n: n\in \mathbb{N}\}\subset \mathfrak{F}_G$, called a
\emph{F\o lner sequence of $G$}, such that
$$\lim_{n\rightarrow \infty} \frac{|g F_n\Delta F_n|}{|F_n|}= 0, \forall g\in G.$$
In the class of countable discrete groups, amenable groups include all solvable groups and groups with subexponential growth.
In the group  $G=\mathbb{Z}$, the sequence $F_n=\{ 0,1,\cdots, n-1\}$ defines a F\o lner sequence, as, indeed, does  $\{ a_n,a_n+1,\cdots,a_n+n- 1 \}$ for any sequence $\{a_n\}_{n\in
\mathbb{N}}\subset \mathbb{Z}$; and in a finite group $G$, if $\{F_n: n\in \mathbb{N}\}$ is a
F\o lner sequence of $G$, then $F_n= G$ once $n$ is large enough.

\medskip

\emph{Throughout this section and the next, we will assume that $G$ is always a countable discrete amenable group.}

\medskip

The well-known Ornstein-Weiss Lemma plays a crucial role in the study of entropy theory for actions of amenable groups \cite{OW} (see also \cite{DZ, HYZ1, RW, We, WZ}).
The following version of it is taken from \cite[1.3.1]{Gronew}.

\begin{proposition} \label{ow-prop-convergence}
Let $f: \mathfrak{F}_G\rightarrow \mathbb{R}$ be a nonnegative function such that $f (E g)= f (E)$ and $f (E\cup F)\le f (E)+ f (F)$ for all $E, F\in \mathfrak{F}_G$ and $g\in G$. Then for any F\o lner sequence $\{F_n: n\in \mathbb{N}\}$ of $G$ the sequence $\left\{\frac{f (F_n)}{|F_n|}: n\in \mathbb{N}\right\}$ converges and the value of the limit is independent of the selection of the F\o lner sequence $\{F_n: n\in \mathbb{N}\}$.
\end{proposition}

\medskip

Let $\mathcal{W}_1, \mathcal{W}_2\in \mathcal{C}_X$. If each element of $\mathcal{W}_1$ is contained in some element of $\mathcal{W}_2$ then we say that $\mathcal{W}_1$ is \emph{finer than $\mathcal{W}_2$} (denoted by $\mathcal{W}_1\succeq \mathcal{W}_2$ or $\mathcal{W}_2\preceq \mathcal{W}_1$).
The join $\mathcal{W}_1\vee \mathcal{W}_2$ is given by
$\mathcal{W}_1\vee \mathcal{W}_2= \{W_1\cap W_2: W_1\in \mathcal{W}_1, W_2\in \mathcal{W}_2\}\in \mathcal{C}_X$,
which extends naturally to a finite collection of covers.
Let $F\in \mathfrak{F}_G$, we set $(\mathcal{W}_1)_F= \bigvee_{g\in F} g^{- 1} \mathcal{W}_1$, and then we consider a nonnegative function $m_{\mathcal{W}_1, \mathcal{W}_2}: \mathfrak{F}_G\rightarrow \mathbb{R}$ given by
 $$m_{\mathcal{W}_1, \mathcal{W}_2} (F)= \max_{K\in (\mathcal{W}_2)_F} \log N ((\mathcal{W}_1)_F, K) \ \text{for each}\ F\in \mathfrak{F}_G.$$

 It is easy to obtain the following useful observation.

 \begin{lemma} \label{welldefine}
 $m_{\mathcal{W}_1, \mathcal{W}_2} (E\cup F)\le m_{\mathcal{W}_1, \mathcal{W}_2} (E)+ m_{\mathcal{W}_1, \mathcal{W}_2} (F)$ for all $E, F\in \mathfrak{F}_G$.
 \end{lemma}
\begin{proof}
Let $E, F\in \mathfrak{F}_G$. From the definition we choose $K\in (\mathcal{W}_2)_{E\cup F}$ with $m_{\mathcal{W}_1, \mathcal{W}_2} (E\cup F)= \log N ((\mathcal{W}_1)_{E\cup F}, K)$.
 Say $K_1\in (\mathcal{W}_2)_E$ and $K_2\in (\mathcal{W}_2)_F$ with $K= K_1\cap K_2$ (no matter if $E$ and $F$ are disjoint), such $K_1$ and $K_2$ must exist. Now let $\mathcal{V}_1\subset (\mathcal{W}_1)_E$ cover $K_1$ with $|\mathcal{V}_1|= N ((\mathcal{W}_1)_E, K_1)$ and let $\mathcal{V}_2\subset (\mathcal{W}_1)_F$ cover $K_2$ with $|\mathcal{V}_2|= N ((\mathcal{W}_1)_F, K_2)$. Obviously we can cover $K_1\cap K_2$ (i.e. $K$) using the family $\mathcal{V}_1\vee \mathcal{V}_2$. Observing that $|\mathcal{V}_1\vee \mathcal{V}_2|\le |\mathcal{V}_1|\cdot |\mathcal{V}_2|$ and each element of $\mathcal{V}_1\vee \mathcal{V}_2$ is contained in some element of $(\mathcal{W}_1)_{E\cup F}$, we have that $N ((\mathcal{W}_1)_{E\cup F}, K)\le |\mathcal{V}_1\vee \mathcal{V}_2|\le N ((\mathcal{W}_1)_E, K_1)\cdot N ((\mathcal{W}_1)_F, K_2)$, which implies the conclusion directly.
\end{proof}

It is easy to check $G$-invariance of the nonnegative function $m_{\mathcal{W}_1, \mathcal{W}_2}: \mathfrak{F}_G\rightarrow \mathbb{R}$. Observing Lemma \ref{welldefine}, we could apply Proposition \ref{ow-prop-convergence} to define
$$h^a (G, \mathcal{W}_1| \mathcal{W}_2)= \lim_{n\rightarrow \infty} \frac{1}{|F_n|} m_{\mathcal{W}_1, \mathcal{W}_2} (F_n)\ge 0,$$
which is independent of the selection of the F\o lner sequence $\{F_n: n\in \mathbb{N}\}$.
Then we define the \emph{topological entropy of $\mathcal{W}_1$} by 
$$h^a (G, \mathcal{W}_1)= h^a (G, \mathcal{W}_1| \{X\})$$
and define the \emph{tail entropy of $(X, G)$ with respect to $\mathcal{W}_2$} by
$$h^a (G, X| \mathcal{W}_2)= \sup_{\mathcal{U}\in \mathcal{C}_X^o} h^a (G, \mathcal{U}| \mathcal{W}_2),$$
and then define the \emph{topological entropy of $(X, G)$}, $h^a (G, X)$, and \emph{tail entropy of $(X, G)$}, $h^{a, *} (G, X)$, respectively, as: 
$$h^a (G, X)= h^a (G, X| \{X\})\ \left(= \sup_{\mathcal{U}\in \mathcal{C}_X^o} h^a (G, \mathcal{U})\right)$$ and
$$h^{a, *} (G, X)= \inf_{\mathcal{V}\in \mathcal{C}_X^o} h^a (G, X| \mathcal{V})\ \left(= \inf_{\mathcal{V}\in \mathcal{C}_X^o} h^a (G, X| \overline{\mathcal{V}})\right)\ge 0.$$

Recalling that in the special case of $G= \mathbb{Z}$ acting on a compact metric space $X$, equivalently, giving a homeomorphism $T: X\rightarrow X$, the above definition recovers the definition given by Misiurewicz in \cite{Mi1}, which was then used to discuss
 weak expansiveness for $\mathbb{Z}$-actions. In fact, Misiurewicz \cite{Mi1} introduced tail entropy in the setting of a compact Hausdorff space $X$ and a continuous transformation of $X$ into itself.

\section{Comparison between sofic and amenable cases}

The following result is the main result of this section, which shows that our definitions of weak expansiveness for actions of sofic groups are equivalent to definitions given in the same spirit of Misiurewicz's ideas when the group is amenable.

\begin{theorem} \label{1306160126}
$(X, G)$ is asymptotically $h$-expansive if and only if $h^{a, *} (G, X)= 0$.
Similarly, $(X, G)$ is $h$-expansive if and only if $h^a (G, X| \mathcal{V})= 0$ for some $\mathcal{V}\in \mathcal{C}_X^o$.
\end{theorem}

Let $Y$ be a finite set, $\{A_i: i\in I\}\subset \{\emptyset\}\cup \mathfrak{F}_Y$ and $\varepsilon, \delta\ge 0$. We say that $\{A_i: i\in I\}$  \emph{$\delta$-covers} $Y$ if
$\left|\bigcup\limits_{i\in I} A_i\right|\ge \delta |Y|$.
$\{A_i: i\in I\}$ are \emph{$\varepsilon$-disjoint} if there exist pairwise disjoint subsets $B_i\subset A_i$ with $|B_i|\ge (1- \varepsilon) |A_i|$ for each $i\in I$.

The next result is the Rokhlin Lemma for sofic approximations of countable discrete groups \cite[Lemma 4.5]{KLAJM}.

\begin{lemma} \label{1107110016}
Let $\Gamma$ be a countable group with the unit $e$ and $0\le \tau< 1, 0< \eta< 1$. Then there are an $l\in \mathbb{N}$ and $\eta', \eta''> 0$ such that, whenever $e\in E_1\subset \cdots\subset E_l$ are finite subsets of $\Gamma$ with $|E_{k- 1}^{- 1} E_k\setminus E_k|\le \eta' |E_k|$ for $k= 2, \cdots, l$, there exists $e\in E\in \mathfrak{F}_\Gamma$ such that for every good enough sofic approximation $\sigma: \Gamma\rightarrow Sym (d)$ for $\Gamma$ with some $d\in \mathbb{N}$ (i.e., $\sigma: \Gamma\rightarrow Sym (d)$ is a map such that
$$\sigma_{s t} (a)= \sigma_s \sigma_t (a), \sigma_s (a)\neq \sigma_{s'} (a), \sigma_e (a)= a$$
for all $a\in B$ with $B\subset \{1, \cdots, d\}$ satisfying $|B|\ge (1- \eta'') d$ and $s, t, s'\in E$ with $s\neq s'$), and any set $V\subset \{1, \cdots, d\}$ with $|V|\ge (1- \tau) d$, there exist $C_1, \cdots, C_l\subset V$ such that
\begin{enumerate}

\item the sets $\sigma (E_k) C_k, k\in \{1, \cdots, l\}$ are pairwise disjoint;

\item $\{\sigma (E_k) C_k: k\in \{1, \cdots, l\}\}$ $(1- \tau- \eta)$-covers $\{1, \cdots, d\}$;

\item $\{\sigma (E_k) c: c\in C_k\}$ is $\eta$-disjoint for each $k\in \{1, \cdots, l\}$; and

\item for every $k\in \{1, \cdots, l\}$ and $c\in C_k$, $E_k\ni s\mapsto \sigma_s (c)$ is bijective.
\end{enumerate}
\end{lemma}

Before proceeding, we also need the following easy observation.

Recall that $\rho$ is a compatible metric on the compact metric space $X$.

\begin{lemma} \label{1107122331}
Let $K\subset X$ be a closed subset and $F\in \mathfrak{F}_G, \mathcal{U}\in \mathcal{C}_X^o$. Then there exists $\delta> 0$ such that
$$K_{F, \delta}= \left\{(x_s)_{s\in F}\in X^F: \max_{s\in F} \rho (x_s, s x)< \delta\ \text{for some}\ x\in K\right\}$$
can be covered by at most $N (\mathcal{U}_F, K)$ elements of $\mathcal{U}^F$.
\end{lemma}
\begin{proof}
Obviously, there exists $\mathcal{V}\subset \mathcal{U}^F$ such that $|\mathcal{V}|\le N (\mathcal{U}_F, K)$ and
\begin{equation*} \label{1107122343}
\cup \mathcal{V}\supset K_F\ \text{where}\ K_F= \{(s x)_{s\in F}: x\in K\}.
\end{equation*}
For example, let $\mathcal{W}\subset \mathcal{U}_F$ such that $|\mathcal{W}|= N (\mathcal{U}_F, K)$ and $\cup \mathcal{W}\supset K$. Now for each $W\in \mathcal{W}$, as $W\in \mathcal{U}_F$, say $W= \bigcap\limits_{s\in F} s^{- 1} U (s)$ with $U (s)\in \mathcal{U}$ for each $s\in F$, we set $\widehat{W}= \prod\limits_{s\in F} U (s)\in \mathcal{U}^F$. Then we can take $\mathcal{V}$ to be $\{\widehat{W}: W\in \mathcal{W}\}$.

Note that $\cup \mathcal{V}$ is an open subset of $X^F$ and $K_F\subset X^F$ is a closed subset. So there exists $\delta> 0$ such that $K_{F, \delta}\subset \cup \mathcal{V}$. This finishes the proof.
\end{proof}

Then, following the ideas of \cite[Lemma 5.1]{KLAJM} we have:

\begin{proposition} \label{1306160111}
Let $\mathcal{U}, \mathcal{U}_2\in \mathcal{C}_X^o$ and $\mathcal{U}_1\in \mathcal{C}_X^c$ with $\mathcal{U}_1\succeq \mathcal{U}_2$. Then
$$h (G, \mathcal{U}| \mathcal{U}_1)\le h^a (G, \mathcal{U}| \overline{\mathcal{U}_2})\ \text{and thus}\ h (G, X| \mathcal{U}_1)\le h^a (G, X| \overline{\mathcal{U}_2}).$$
\end{proposition}
\begin{proof}
Let $\varepsilon> 0$. We choose $1> \eta> 0$ small enough such that
\begin{equation} \label{1107130122}
\frac{h^a (G, \mathcal{U}| \overline{\mathcal{U}_2})+ \varepsilon}{1- \eta}+ 2 \eta \log |\mathcal{U}|\le h^a (G, \mathcal{U}| \overline{\mathcal{U}_2})+ 2 \varepsilon
\end{equation}
and $K\in \mathfrak{F}_G, \delta'> 0$ such that, once $F\in \mathfrak{F}_G$ satisfies $|K F\Delta F|\le \delta' |F|$ then
\begin{equation} \label{6.star}
\frac{1}{|F|} \max_{K\in (\overline{\mathcal{U}_2})_F} \log N (\mathcal{U}_F, K)\le h^a (G, \mathcal{U}| \overline{\mathcal{U}_2})+ \varepsilon.
\end{equation}

Observing that $\mathcal{U}_2\in \mathcal{C}_X^o$ and $\mathcal{U}_1\in \mathcal{C}_X^c$ satisfy $\mathcal{U}_1\succeq \mathcal{U}_2$, we can choose $\delta''> 0$ such that, for each $U_1\in \mathcal{U}_1$ there exists $U_2\in \mathcal{U}_2$ containing the open $\delta''$-neighborhood of $U_1$. Now let $l\in \mathbb{N}$ and $\eta'> 0$ be as given by Lemma \ref{1107110016} with respect to $\tau= \eta$ and $\eta$. In $\mathfrak{F}_G$ we take $e_G\in F_1\subset \cdots\subset F_l$ such that $|F_{k- 1}^{- 1} F_k\setminus F_k|\le \eta' |F_k|$ for each $k= 2, \cdots, l$ and $|K F_k\Delta F_k|\le \delta' |F_k|$ for all $k= 1, \cdots, l$. As the group $G$ is amenable, such subsets $F_1, \cdots, F_l$ must exist. Thus, by \eqref{6.star},
\begin{equation} \label{1107122326}
\max_{k= 1}^l \frac{1}{|F_k|} \max_{K\in (\overline{\mathcal{U}_2})_{F_k}} \log N (\mathcal{U}_{F_k}, K)\le h^a (G, \mathcal{U}| \overline{\mathcal{U}_2})+ \varepsilon.
\end{equation}
For each $k= 1, \cdots, l$ and any $K\in (\overline{\mathcal{U}_2})_{F_k}$, let $\delta (k, K)> 0$ be as given by Lemma \ref{1107122331} with respect to $K, F_k$ and $\mathcal{U}$, and then set $\delta_k= \min \{\delta (k, K): K\in (\overline{\mathcal{U}_2})_{F_k}\}$.
 We take $\delta> 0$ such that $\delta\le \min\left\{(\delta'')^2, \delta_1^2, \cdots, \delta_l^2, \frac{\eta}{|F_l|}\right\}$ and
\begin{equation} \label{1107130020}
\sum_{j= 0}^{[|F_l| \delta d]} \binom{d}{j}< (1+ \varepsilon)^d\ \text{for all large enough}\ d\in \mathbb{N}.
\end{equation}

Now let $\sigma: G\rightarrow Sym (d)$ be a good enough sofic approximation for $G$ with some $d\in \mathbb{N}$ (and hence $d\in \mathbb{N}$ is large enough). If $(x_1, \cdots, x_d)\in X^d_{F_l, \delta, \sigma}$ then
$$\max_{s\in F_l} \sqrt{\sum_{i= 1}^d \frac{1}{d} \rho^2 (s x_i, x_{\sigma_s (i)})}< \delta,$$
which implies that $|J (x_1, \cdots, x_d, F_l)|\ge (1- |F_l| \delta) d$, where
$$J (x_1, \cdots, x_d, F_l)= \left\{i\in \{1, \cdots, d\}: \max_{s\in F_l} \rho (s x_i, x_{\sigma_s (i)})< \sqrt{\delta}\right\}.$$
Denote by $\Theta$ the set of all subsets of $\{1, \cdots, d\}$ with at least $(1- |F_l| \delta) d$ many elements, and for each $\theta\in \Theta$ denote by $X^d_{F_l, \delta, \sigma, \theta}$ the set of all $(x_1, \cdots, x_d)\in X^d_{F_l, \delta, \sigma}$ with $J (x_1, \cdots, x_d, F_l)= \theta$. Then
\begin{equation} \label{1107130014}
|\Theta|= \sum_{j= 0}^{[|F_l| \delta d]} \binom{d}{j}< (1+ \varepsilon)^d\ (\text{using \eqref{1107130020}})\ \text{and}\ \bigcup_{\theta\in \Theta} X^d_{F_l, \delta, \sigma, \theta}= X^d_{F_l, \delta, \sigma}.
\end{equation}

Let $\theta\in \Theta$. As $\sigma$ is good enough, by Lemma \ref{1107110016} there exist $C_1, \cdots, C_l\subset \theta$ with
\begin{enumerate}

\item the sets $\sigma (F_k) C_k, k\in \{1, \cdots, l\}$ are pairwise disjoint;

\item $\{\sigma (F_k) c: c\in C_k\}$ is $\eta$-disjoint for each $k= 1, \cdots, l$;

\item $\{\sigma (F_k) C_k: k\in \{1, \cdots, l\}\}$ $(1- 2 \eta)$-covers $\{1, \cdots, d\}$; and

\item for every $k\in \{1, \cdots, l\}$ and $c\in C_k$, $F_k\ni s\mapsto \sigma_s (c)$ is bijective.
\end{enumerate}
Set $J_\theta= \{1, \cdots, d\}\setminus \cup \{\sigma (F_k) C_k: k\in \{1, \cdots, l\}\}$. Then
\begin{equation} \label{1107130110}
|J_\theta|\le 2 \eta d\ \text{and}\ \sum_{k= 1}^l |F_k|\cdot |C_k|\le \frac{1}{1- \eta} \sum_{k= 1}^l |\sigma (F_k) C_k|\le \frac{d}{1- \eta}.
\end{equation}
Now let $W\in (\mathcal{U}_1)^d$, say $W= \prod_{i= 1}^d U_1^{(i)}$ and for each $i= 1, \cdots, d$ the open $\delta''$-neighborhood of $U_1^{(i)}$ is contained in $U_2^{(i)}\in \mathcal{U}_2$. For each $k= 1, \cdots, l$ and any $c_k\in C_k$, as $C_k\subset \theta$ and $F_k\subset F_l$, if $(x_1, \cdots, x_d)\in X^d_{F_l, \delta, \sigma, \theta}\cap W$, then
$$\max_{s\in F_k} \rho (x_{\sigma_s (c_k)}, s x_{c_k})< \sqrt{\delta}\le \delta'',$$
and so $s x_{c_k}\in U_2^{(\sigma_s (c_k))}$ (as $x_{\sigma_s (c_k)}\in U_1^{(\sigma_s (c_k))}$) for all $s\in F_k$ by the selection of $\delta''$, thus
$$x_{c_k}\in \bigcap_{s\in F_k} s^{- 1} \overline{U_2^{(\sigma_s (c_k))}}\ (\text{denoted by}\ Q)\in (\overline{\mathcal{U}_2})_{F_k},$$
which implies by applying Lemma \ref{1107122331} and \eqref{1107122326} that we can cover
\begin{eqnarray*}
& & \{(x_i)_{i\in \sigma (F_k) c_k}: (x_1, \cdots, x_d)\in X^d_{F_l, \delta, \sigma, \theta}\cap W\} \\
&\subset & \left\{(x_i)_{i\in \sigma (F_k) c_k}: \max_{s\in F_k} \rho (x_{\sigma_s (c_k)}, s x)< \delta_k\ \text{for some}\ x\in Q\right\}
\end{eqnarray*}
by at most (observing the selection of $\delta_k$)
$$N (\mathcal{U}_{F_k}, Q)\le e^{|F_k|\cdot [h^a (G, \mathcal{U}| \overline{\mathcal{U}_2})+ \varepsilon]}$$
 elements of $\mathcal{U}^{\sigma (F_k) c_k}$, and so it is not hard to cover
$$\{(x_i)_{i\in \sigma (F_k) C_k}: (x_1, \cdots, x_d)\in X^d_{F_l, \delta, \sigma, \theta}\cap W\}$$
using at most
$$e^{|C_k|\cdot |F_k|\cdot [h^a (G, \mathcal{U}| \overline{\mathcal{U}_2})+ \varepsilon]}$$
 elements of $\mathcal{U}^{\sigma (F_k) C_k}$. Thus
\begin{eqnarray} \label{1107131006}
\log N (\mathcal{U}^d, X^d_{F_l, \delta, \sigma, \theta}\cap W)&\le & \sum_{k= 1}^l |C_k|\cdot |F_k|\cdot [h^a (G, \mathcal{U}| \overline{\mathcal{U}_2})+ \varepsilon]+ |J_\theta| \log |\mathcal{U}|\nonumber \\
&\le & d \left(\frac{h^a (G, \mathcal{U}| \overline{\mathcal{U}_2})+ \varepsilon}{1- \eta}+ 2 \eta \log |\mathcal{U}|\right)\ (\text{using \eqref{1107130110}}) \nonumber \\
&\le & d (h^a (G, \mathcal{U}| \overline{\mathcal{U}_2})+ 2 \varepsilon)\ (\text{using \eqref{1107130122}}).
\end{eqnarray}

Combining \eqref{1107130014} with \eqref{1107131006} we obtain
$$\log N (\mathcal{U}^d, X^d_{F_l, \delta, \sigma}\cap W)\le d (h^a (G, \mathcal{U}| \overline{\mathcal{U}_2})+ 2 \varepsilon+ \log (1+ \varepsilon)).$$
By the arbitrariness of $\varepsilon> 0$ and $W\in (\mathcal{U}_1)^d$ we obtain the conclusion.
\end{proof}

We also have \cite[Lemma 4.6]{KLAJM}, an improved version of Lemma \ref{1107110016} for an amenable group. Recall that the group $G$ is amenable throughout the whole section.

\begin{lemma} \label{1107112242}
Let $0\le \tau< 1, 0< \eta< 1$ and $K\in \mathfrak{F}_G, \delta> 0$. Then there are an $l\in \mathbb{N}$ and $F_1, \cdots, F_l\in \mathfrak{F}_G$ with $|K F_k\setminus F_k|< \delta |F_k|$ and $|F_k K\setminus F_k|< \delta |F_k|$ for all $k= 1, \cdots, l$, such that for every good enough sofic approximation $\sigma: G\rightarrow Sym (d)$ for $G$ with some $d\in \mathbb{N}$ and any set $V\subset \{1, \cdots, d\}$ with $|V|\ge (1- \tau) d$, there exist $C_1, \cdots, C_l\subset V$ such that
\begin{enumerate}

\item the sets $\sigma (F_k) C_k, k\in \{1, \cdots, l\}$ are pairwise disjoint;

\item $\{\sigma (F_k) C_k: k\in \{1, \cdots, l\}\}$ $(1- \tau- \eta)$-covers $\{1, \cdots, d\}$; and

\item for every $k\in \{1, \cdots, l\}$, the map $F_k\times C_k\ni (s, c)\mapsto \sigma_s (c)$ is bijective.
\end{enumerate}
\end{lemma}

Let $\mathcal{U}\in \mathcal{C}_X, \varepsilon> 0$ and $F\in \mathfrak{F}_G, \delta> 0$. We set
$$h_{F, \delta} (G, \varepsilon| \mathcal{U})= \limsup_{i\rightarrow \infty} \frac{1}{d_i} \log \max_{V\in \mathcal{U}^{d_i}} N_\varepsilon \left(X^{d_i}_{F, \delta, \sigma_i}\cap V, \rho_{d_i}\right)\ (\text{recalling \eqref{metric}}),$$
$$h (G, \varepsilon| \mathcal{U})= \inf_{F\in \mathfrak{F}_G} \inf_{\delta> 0} h_{F, \delta} (G, \varepsilon| \mathcal{U}).$$
Now let $\mathcal{V}\in \mathcal{C}_X$ and $\varepsilon_1, \varepsilon_2> 0$. Assume that $\text{diam} (\mathcal{V})< \varepsilon_1$ and any open ball with radius $\varepsilon_2$ is contained in some element of $\mathcal{V}$. It is easy to obtain
$$h_{F, \delta} (G, \varepsilon_1| \mathcal{U})\le h_{F, \delta} (G, \mathcal{V}| \mathcal{U})\le h_{F, \delta} (G, \varepsilon_2| \mathcal{U}),$$
$$h (G, \varepsilon_1| \mathcal{U})\le h (G, \mathcal{V}| \mathcal{U})\le h (G, \varepsilon_2| \mathcal{U}).$$
Thus we have
\begin{equation} \label{1306210043}
h (G, X| \mathcal{U})= \lim_{\varepsilon\rightarrow 0} h (G, \varepsilon| \mathcal{U})= \sup_{\varepsilon> 0} h (G, \varepsilon| \mathcal{U}).
\end{equation}

Following the ideas of \cite[Lemma 5.2]{KLAJM} we have:

\begin{proposition} \label{1306160116}
Let $\mathcal{U}\in \mathcal{C}_X$. Then $h (G, X| \mathcal{U})\ge h^a (G, X| \mathcal{U})$.
\end{proposition}
\begin{proof}
Let $\mathcal{U}_1\in \mathcal{C}_X^o$. We only need to prove $h (G, X| \mathcal{U})\ge h^a (G, \mathcal{U}_1| \mathcal{U})$.

We choose $\varepsilon> 0$ such that any open ball with $\rho$-radius $\varepsilon$ is contained in some element of $\mathcal{U}_1$, and let $\theta> 0, F\in \mathfrak{F}_G, \delta> 0$. We are to finish the proof by showing
\begin{equation} \label{1107111631}
\frac{1}{d} \log \max_{V\in \mathcal{U}^{d}} N_\varepsilon \left(X^{d}_{F, \delta, \sigma}\cap V, \rho_{d}\right)\ge h^a (G, \mathcal{U}_1| \mathcal{U})- 3 \theta
\end{equation}
 once $\sigma: G\rightarrow Sym (d)$ is a good enough sofic approximation for $G$ with some $d\in \mathbb{N}$.

 Let $M> 0$ be large enough and $\delta'> 0$ small enough such that the diameter of the space $X$ is at most $M$ and
\begin{equation} \label{1107111645}
\sqrt{\delta'} M< \frac{\delta}{2}, \delta' \log |\mathcal{U}|< \theta\ \text{and}\ (1- \delta') h^a (G, \mathcal{U}_1| \mathcal{U})\ge h^a (G, \mathcal{U}_1| \mathcal{U})- \theta.
\end{equation}

Applying Lemma \ref{1107112242}, there are an $l\in \mathbb{N}$ and $F_1, \cdots, F_l\in \mathfrak{F}_G$ so that
\begin{equation} \label{1107120045}
\min_{k= 1}^l \min_{s\in F} \frac{|s^{- 1} F_k\cap F_k|}{|F_k|}\ge 1- \delta'
\end{equation}
and
\begin{equation} \label{1107112321}
\min_{k= 1}^l \frac{1}{|F_k|} \max_{K\in \mathcal{U}_{F_k}} \log N ((\mathcal{U}_1)_{F_k}, K)\ge \max \{0, h^a (G, \mathcal{U}_1| \mathcal{U})- \theta\},
\end{equation}
such that once $\sigma: G\rightarrow Sym (d)$ is a good enough sofic approximation for $G$ with some $d\in \mathbb{N}$ then there exist $C_1, \cdots, C_l\subset \{1, \cdots, d\}$ satisfying
\begin{enumerate}

\item the sets $\sigma (F_k) C_k, k\in \{1, \cdots, l\}$ are pairwise disjoint;

\item $\{\sigma (F_k) C_k: k\in \{1, \cdots, l\}\}$ $(1- \delta')$-covers $\{1, \cdots, d\}$;

\item for every $k\in \{1, \cdots, l\}$, the map $F_k\times C_k\ni (s, c)\mapsto \sigma_s (c)$ is bijective; and

\item for all $k\in \{1, \cdots, l\}$ and $s\in F, s_k\in F_k, c_k\in C_k$, $\sigma_{s s_k} (c_k)= \sigma_s \sigma_{s_k} (c_k)$.
\end{enumerate}
Remark again that since the group $G$ is amenable, such subsets $F_1, \cdots, F_l$ must exist.

Now assume that $\sigma: G\rightarrow Sym (d)$ is a good enough sofic approximation for $G$ with some $d\in \mathbb{N}$ and let $C_1, \cdots, C_l\subset \{1, \cdots, d\}$ be constructed as above.

\medskip

For each $k\in \{1, \cdots, l\}$ and any $K\in \mathcal{U}_{F_k}$, we take a maximal $(\rho_{F_k}, \varepsilon)$-separated subset $E_{k, K}$ of $K$ which is obviously finite, where
$$\rho_{F_k} (x, x')= \max_{g\in F_k} \rho (g x, g x')\ \text{for}\ x, x'\in X.$$
Then for each $y\in K$ there exists an $x\in E_{k, K}$ such that $\rho_{F_k} (x, y)< \varepsilon$. Observe that any open ball with $\rho$-radius $\varepsilon$ is contained in some element of $\mathcal{U}_1$, and hence any open ball with $\rho_{F_k}$-radius $\varepsilon$ is contained in some element of $(\mathcal{U}_1)_{F_k}$, thus
\begin{equation} \label{1306220131}
|E_{k, K}|\ge N ((\mathcal{U}_1)_{F_k}, K).
\end{equation}

Now let $(y_1, \cdots, y_l)$ be any $l$-tuple with
$$y_k\in \prod_{c\in C_k} E_{k, K_c},\ \text{where}\ k\in \{1, \cdots, l\}\ \text{and}\ K_c\in \mathcal{U}_{F_k}\ \text{for each}\ c\in C_k.$$
From the construction of $C_1, \cdots, C_l$, it is not hard to see that there exists at least one point $(x_1, \cdots, x_d)\in X^d$ such that once $i\in \sigma (F_k) C_k$ for some $k\in \{1, \cdots, l\}$, say $i= \sigma_{s_k} (c_k)$ with $s_k\in F_k$ and $c_k\in C_k$, then $x_i= s_k y_k (c_k)$.
Let $(x_1, \cdots, x_d)\in X^d$ be such a point (corresponding to the $l$-tuple $(y_1, \cdots, y_l)$).
Let $s\in F$ and $i\in \{1, \cdots, d\}$. Once $i= \sigma_{s_k} (c_k)$ for some $s_k\in F_k$ and $c_k\in C_k, k\in \{1, \cdots, l\}$, if $s s_k\in F_k$, then $s x_i= s s_k y_k (c_k)= x_{\sigma_{s s_k} (c_k)}= x_{\sigma_s \sigma_{s_k} (c_k)}= x_{\sigma_s (i)}$. Which implies that
\begin{equation} \label{1107120038}
\frac{1}{d} \sum_{i= 1}^d \rho^2 (s x_i, x_{\sigma_s (i)})= \frac{1}{d} \sum_{i\in \{1, \cdots, d\}\setminus E} \rho^2 (s x_i, x_{\sigma_s (i)})\le \frac{M^2}{d} |\{1, \cdots, d\}\setminus E|,
\end{equation}
where
$$E= \bigcup_{k= 1}^l \sigma (s^{- 1} F_k\cap F_k) C_k.$$
Using the construction of $C_1, \cdots, C_l$ again, by \eqref{1107120045} one has
\begin{equation} \label{1107120043}
|E|= \sum_{k= 1}^l |s^{- 1} F_k\cap F_k|\cdot |C_k|\ge (1- \delta') \sum_{k= 1}^l |F_k|\cdot |C_k|\ge d (1- \delta')^2\ge d (1- 2 \delta').
\end{equation}
Combining \eqref{1107120038} with \eqref{1107120043}, we obtain
$$\frac{1}{d} \sum_{i= 1}^d \rho^2 (s x_i, x_{\sigma_s (i)})\le 2 \delta' M^2< \frac{\delta^2}{2}$$
by the selection \eqref{1107111645} of $M$ and $\delta'$. In particular, $(x_1, \cdots, x_d)\in X^d_{F, \delta, \sigma}$.

\medskip

On one hand, from the above constructions, it is easy to check that, given $K_c\in \mathcal{U}_{F_k}$ for each $k\in \{1, \cdots, l\}$ and any $c\in C_k$, there exists $W\in \mathcal{W}$ with
$$\mathcal{W}\doteq \prod_{k= 1}^l \mathcal{U}^{\sigma (F_k) C_k}\times \{X\}^{\{1, \cdots, d\}\setminus \bigcup\limits_{k= 1}^l \sigma (F_k) C_k},$$
such that if a point $(x_1, \cdots, x_d)\in X^d$ corresponds to an $l$-tuple from
$\prod_{k= 1}^l \prod_{c\in C_k} E_{k, K_c}$
then $(x_1, \cdots, x_d)\in W$. 
On the other hand, assume that $(x_1, \cdots, x_d)\in X^d$ and $(x_1', \cdots, x_d')\in X^d$ correspond to distinct $l$-tuples $(y_1, \cdots, y_l)$ and $(y_1', \cdots, y_l')$ from
$\prod_{k= 1}^l \prod_{c\in C_k} E_{k, K_c}$, where $K_c\in \mathcal{U}_{F_k}$ for each $k\in \{1, \cdots, l\}$ and any $c\in C_k$,
respectively. Thus, for some $k\in \{1, \cdots, l\}$ and $c\in C_k$, $y_k (c)$ and $y_k' (c)$ are distinct elements of $E_{k, K_c}$, in particular, $\rho_{F_k} (y_k (c), y_k' (c))\ge \varepsilon$, and then
\begin{eqnarray*}
\rho_d ((x_1, \cdots, x_d), (x_1', \cdots, x_d'))&\ge & \max_{i\in \sigma (F_k) c} \rho (x_i, x_i') \\
&= & \max_{s\in F_k} \rho (s y_k (c), s y_k' (c))= \rho_{F_k} (y_k (c), y_k' (c))\ge \varepsilon.
\end{eqnarray*}
This implies that
\begin{eqnarray*}
\max_{W\in \mathcal{W}} N_\varepsilon \left(X^{d}_{F, \delta, \sigma}\cap W, \rho_{d}\right)&\ge & \left|\prod_{k= 1}^l \prod_{c\in C_k} E_{k, K_c}\right|= \prod_{k= 1}^l \prod_{c\in C_k} |E_{k, K_c}| \\
&\ge & \prod_{k= 1}^l \prod_{c\in C_k}
N ((\mathcal{U}_1)_{F_k}, K_c)\ (\text{using \eqref{1306220131}}).
\end{eqnarray*}
Combining the above estimation with the fact that $\{\sigma (F_k) C_k: k\in \{1, \cdots, l\}\}$ $(1- \delta')$-covers $\{1, \cdots, d\}$, we obtain that
\begin{eqnarray} \label{1306220151}
& &\hskip - 66pt \log \left[|\mathcal{U}|^{\delta' d} \max_{V\in \mathcal{U}^{d}} N_\varepsilon \left(X^{d}_{F, \delta, \sigma}\cap V, \rho_{d}\right)\right]\nonumber \\
&\ge & \log \max_{W\in \mathcal{W}} N_\varepsilon \left(X^{d}_{F, \delta, \sigma}\cap W, \rho_{d}\right)\nonumber \\
&\ge & \sum_{k= 1}^l |C_k| \max_{K\in \mathcal{U}_{F_k}} \log N ((\mathcal{U}_1)_{F_k}, K)\nonumber \\
&\ge & \sum_{k= 1}^l |C_k|\cdot |F_k|\cdot \max \{0, h^a (G, \mathcal{U}_1| \mathcal{U})- \theta\}\ (\text{using \eqref{1107112321}})\nonumber \\
&\ge & d (1- \delta') \max \{0, h^a (G, \mathcal{U}_1| \mathcal{U})- \theta\}\nonumber \\
&\ge & d [h^a (G, \mathcal{U}_1| \mathcal{U})- 2 \theta]\ (\text{using \eqref{1107111645}}).
\end{eqnarray}
Then \eqref{1107111631} follows from \eqref{1107111645} and \eqref{1306220151}.
\end{proof}

Now let us prove our main result of this section.

\begin{proof}[Proof of Theorem \ref{1306160126}]
Assume that $(X, G)$ is asymptotically $h$-expansive. Let $\varepsilon> 0$. Then there exists $\mathcal{U}_1\in \mathcal{C}_X^o$ such that $h (G, X| \mathcal{U}_1)< \varepsilon$. Thus $h^{a, *} (G, X)\le h^a (G, X| \mathcal{U}_1)< \varepsilon$ by Proposition \ref{1306160116}, and finally $h^{a, *} (G, X)= 0$.

Now assume that $h^{a, *} (G, X)= 0$ and let $\varepsilon> 0$. By the definition, there exists $\mathcal{U}_3\in \mathcal{C}_X^o$ with $h^a (G, X| \overline{\mathcal{U}_3})< \varepsilon$. As $X$ is a compact metric space, we can take $\mathcal{U}_4\in \mathcal{C}_X^o$ with $\overline{\mathcal{U}_4}\succeq \mathcal{U}_3$, then $h (G, X| \overline{\mathcal{U}_4})< \varepsilon$ by Proposition \ref{1306160111}, and so $h^* (G, X)\le h (G, X| \mathcal{U}_4)< \varepsilon$, thus $h^* (G, X)\le 0$. That is, $(X, G)$ is asymptotically $h$-expansive.

We could prove similarly the remaining part of the theorem.
\end{proof}

By the same proof we have:

\begin{theorem} \label{1409221729}
$h^* (G, X)= h^{a, *} (G, X)$.
\end{theorem}

Setting $\mathcal{U}_1= \mathcal{U}_2= \{X\}$ in Proposition \ref{1306160111} and $\mathcal{U}= \{X\}$ in Proposition \ref{1306160116}, we obtain directly the following observation \cite[Theorem 5.3]{KLAJM}.

\begin{corollary} \label{1307112259}
$h (G, X)= h^a (G, X)$.
\end{corollary}

\section{Actions over compact metric groups}

In this section we shall provide more interesting asymptotically $h$-expansive examples when we consider actions of countable discrete amenable groups.

\medskip

Recall that any $C^\infty$ diffeomorphism on a compact
manifold is asymptotically $h$-expansive by Buzzi \cite{Bu}. Moreover, if we consider a differentiable action $(X, G)$ in the sense that the homeomorphism of $X$ given by each $g\in G$ is a $C^{(1)}$ map, where $X$ is a compact smooth manifold (here we allow a smooth manifold to have different dimensions for different connected components, even including zero dimension) and $G$ is a countable discrete amenable group containing $\mathbb{Z}$ as a subgroup of infinite index, then the action $(X, G)$ has zero topological entropy (see for example \cite[Lemma 5.7]{LT} by Li and Thom), and so it is $h$-expansive.

\medskip

Moreover, inspired by Misiurewicz's work \cite[\S 7]{Mi1}, in the following we prove:

 \begin{theorem} \label{1307122130}
 Let $G$ be a countable discrete amenable group acting on a compact metric group $X$ by continuous automorphisms. Then the action $(X, G)$ is asymptotically $h$-expansive if and only if  $h^a (G, X)< \infty$.
 \end{theorem}

 \begin{remark} \label{1409270036}
Let $G$ be a countable discrete amenable group acting on a compact metrizable group $X$ by continuous automorphisms.
If $(X, G)$ has finite topological entropy then, by Corollary \ref{1310281312}, Corollary \ref{1307112259} and Theorem \ref{1307122130},
 it admits an invariant measure with maximal entropy. If the amenable group $G$ is infinite and $(X, G)$ has infinite topological entropy, then, by the affinity property of measure-theoretic entropy (for actions of a countable discrete infinite amenable group), it will also admit an invariant measure with maximal entropy. Indeed,
 if we let $\mu$ be the normalized Haar measure of the compact metric group $X$ (and hence $\mu$ is automatically $G$-invariant), then
 the measure-theoretic $\mu$-entropy coincides with the topological entropy of the system \cite[Theorem 2.2]{Den} (remark that $X$ was assumed to be abelian in \cite[Theorem 2.2]{Den}, which is in fact not needed).
 Note that this fact was first pointed out by Berg \cite{Berg} in the case of $G= \mathbb{Z}$.
\end{remark}

\begin{remark}
The conclusion of Theorem \ref{1307122130} does not hold if we remove the group structure from $X$ even for the $\Zb$-actions. There are many such actions. The first one may be the Gurevic's example (for the detailed construction see for example \cite[Page 192]{Walters82}), which is a $\Zb$-action having finite entropy and without any invariant measure attaining maximal entropy. Another example is \cite[Example 6.4]{Mi1}, which is a $\mathbb{Z}$-action with finite entropy such that it is not asymptotically $h$-expansive, while each of its invariant measures has maximal entropy.
\end{remark}

The ``only if'' part of Theorem \ref{1307122130} comes from Corollary \ref{1207122131} and Corollary \ref{1307112259}. The proof of its ``if'' part relies on the concept of homogeneous measures.
And so first we recall the definition of homogeneous measures following \cite[\S 7]{Mi1}, which was first introduced by R. Bowen in \cite{B71}.

Let $G$ be a group acting on a compact metric space $X$, and denote by $\alpha$ the action. Let $\mu\in M (X)$. For each $\cU\in \cC^o_X$, we set
 $$P(\cU)=\bigcup_{x\in X}\left\{U\in\cU:x\in U\ \text{and}\ \mu(U)=\max_{V\in \cU, x\in V} \mu (V)\right\}\in \cC^o_X.$$
The measure $\mu$ is called \emph{$\alpha$-homogeneous} if there exist mappings $D:\cC^o_X\rightarrow \cC^o_X$ and $c: \cC^o_X\rightarrow (0,\infty)$ such that for any $\cU\in \cC^o_X$ and each $F\in\cF_G$ we have
 $$\mu(V)\leq c(\cU)\mu(U)\ \text{for all}\ U\in P(\cU_F)\ \text{and}\ V\in D(\cU)_F.$$

 In general, it is not easy to check if a measure is homogeneous. While for $G= \Zb$ Misiurewicz gave a sufficient condition for the existence of such a measure \cite[Theorem 7.2]{Mi1}, which can be generalized to a general group $G$ as follows.

 \medskip

Following the spirit of Misiurewicz \cite{Mi1}, let $H$ be a group acting on a compact metric space $(X, \rho)$ with the action $\Phi$, that is, $\Phi$ is a homomorphism of $H$ into the group of all homeomorphisms of $X$. Recall that $\Phi$ is {\it transitive} if for any $x, y\in X$ there exists $g\in H$ with $g x= y$, and {\it equicontinuous} if for each $\varepsilon> 0$ there exists $\delta> 0$ such that $\rho (x_1, x_2)< \delta$ implies $\rho (g x_1, g x_2)< \varepsilon$ for all $x_1, x_2\in X$ and $g\in H$. Recall that $\mu\in M (X)$ is {\it invariant with respect to $\Phi$} if $g \mu= \mu$ for each $g\in H$. While the transitivity here is different from the usual one in topological dynamics, the definitions of equicontinuity and invariance of a measure are just as usual.

\begin{lemma}
\label{L-homogeneous}
Let $G$ be a group acting on a compact metric space $X$, and denote by $\alpha$ the action. Let $H$ be a group and $T$ an action of $G$ on $H$ by homomorphisms $g\mapsto T_g$ for each $g\in G$, equivalently, each $T_g: H\rightarrow H$ is a homomorphism. Now assume that $\Phi$ is a transitive equicontinuous action of $H$ on $X$ by homeomorphisms such that $\mu\in M (X)$ is invariant with respect to $\Phi$ and $g (h x)= (T_g h) (g x)$ for all $x\in X, g\in G$ and $h\in H$. Then $\mu$ is $\alpha$-homogeneous.
\end{lemma}
\begin{proof}
Let $\cU\in \cC^o_X$. We aim to construct some $D (\cU)\in \cC_X^o$ with $c (\cU)= 1$ such that both of mappings $D$ and $c$ satisfy the condition of $\mu$ being $\alpha$-homogeneous.

Let $\varepsilon> 0$ be a Lebesgue number of $\cU$. By the equicontinuity of $\Phi$ there exists $\delta> 0$ such that $\rho (x_1, x_2)\le \delta$ implies $\rho (h x_1, h x_2)< \frac{\varepsilon}{3}$ for all $x_1, x_2\in X$ and $h\in H$.
Now let $W$ be
 a non-empty open set with diameter at most $\delta$.
  For each $x\in X$, denote by $W_x$ the set of all $y\in X$ such that both $x$ and $y$ are contained in $h W$ for some $h\in H$, then $\rho (x, y)< \frac{\varepsilon}{3}$ for each $y\in W_x$, and so $W_x\subset U_x$ for some $U_x\in\cU$.

Observe that the family of non-empty open subsets $\{h W: h\in H\}$ covers $X$ by the transitivity of the action $\Phi$, and so we could choose $D(\cU)\in \cC_X^o$ with $D (\cU)\subset \{h W: h\in H\}$. Now we show that $D (\cU)$ and $c (\cU)= 1$ satisfy the required properties.

 Let $F\in \cF_G$.
  Set $W_{F, x}=\bigcap_{\gamma\in F} \gamma^{- 1} W_{\gamma x}\subset \bigcap_{\gamma\in F} \gamma^{- 1} U_{\gamma x}\in \cU_F$ for each $x\in X$. Let $U\in D(\cU)_F$, and say $U= \bigcap_{g\in F} g^{- 1} (h_g W)$ with $h_g\in H$ for any $g\in F$. Then $g x\in h_g W$ and hence $h_g W\subset W_{g x}$ for all $x\in U$ and $g\in F$.
   Thus, for any $x\in U$,
   $$U= \bigcap_{g\in F} g^{- 1} (h_g W)\subset \bigcap_{g\in F} g^{- 1} W_{g x}= W_{F, x}\ni x$$
   and hence (observing $W_{F, x}\subset \bigcap_{\gamma\in F} \gamma^{- 1} U_{\gamma x}\in \cU_F$)
   \begin{equation}
\label{e1}
\mu(U)\leq \mu (W_{F, x})\leq \max\{\mu(V): x\in V\ \text{and}\ V\in \cU_F\}.
\end{equation}

Let $x\in X$ and $k\in H$. For $y\in X$, $y\in k W_x$, equivalently $k^{- 1} y\in W_x$, if and only if both $x$ and $k^{- 1} y$ are contained in $h W$ for some $h\in H$, equivalently both $k x$ and $y$ are contained in $k h W$ for some $h\in H$, if and only if $y\in W_{k x}$. Thus $k W_x= W_{k x}$.
Now let $h\in H, Y\subset X$ and $g\in G$. By assumptions, $g (h (g^{- 1} Y))= (T_g h) (g (g^{- 1} Y))= (T_g h) Y$, equivalently, $h (g^{- 1} Y)= g^{- 1} ((T_g h) Y)$.
And hence
  $$h (g^{- 1} W_{g x})= g^{- 1} ((T_g h) W_{g x})= g^{- 1} (W_{(T_g h) (g x)})= g^{- 1} (W_{g (h x)}).$$
  Which implies that, for all $x\in X$ and any $h\in H, F\in \cF_G$,
$$h W_{F, x}= h \left(\bigcap_{\gamma\in F} \gamma^{- 1} W_{\gamma x}\right)= \bigcap_{\gamma\in F} h (\gamma^{- 1} W_{\gamma x})= \bigcap_{\gamma\in F} \gamma^{- 1} W_{\gamma (h x)}= W_{F, h x}.$$
As the action $\Phi$ is transitive and $\mu$ is invariant with respect to $\Phi$, we obtain
\begin{equation} \label{e2}
\mu(W_{F, x})=\mu(W_{F, y})\ \text{for all}\ x, y\in X\ \text{and any}\ F\in \cF_G.
\end{equation}
Summing up, for each $F\in \cF_G$ and any $U\in D (\cU)_F$ and $Q\in P (\cU_F)$, let $x\in U$, and by the construction of $P (\cU_F)$ we could choose $y\in Q$ with $\mu (Q)= \max \{\mu (V): y\in V\ \text{and}\ V\in \cU_F\}$, and then
\begin{eqnarray*}
\mu (U)&\le & \mu (W_{F, x})\ (\text{using \eqref{e1}})= \mu (W_{F, y})\ (\text{using \eqref{e2}}) \\
&\le & \max\{\mu(V): y\in V\ \text{and}\ V\in \cU_F\}\ (\text{using \eqref{e1} again})= \mu (Q).
\end{eqnarray*}
Recalling $c(\cU)=1$, we see that the measure $\mu$ is $\alpha$-homogeneous.
\end{proof}

As a direct corollary, we have:

\begin{proposition} \label{1307122312}
Let $G$ be a group acting on a compact metric group $X$ by continuous automorphisms, and denote by $\alpha$ the action. Let $\mu\in M (X)$ be the normalized Haar measure of the group $X$. Then $\mu$ is $\alpha$-homogeneous.
 \end{proposition}
 \begin{proof}
 Set $H= X$ and let $T$ be the action $\alpha$ of $G$ on $H$ by automorphisms $T_g: H\rightarrow H, h\mapsto g h$.
  Now let $\Phi$ be the action of $H$ on $X$ by homeomorphisms $h: X\rightarrow X, x\mapsto h x$ for all $h\in X$. Obviously, the action $\Phi$ is transitive and the measure $\mu$ is invariant with respect to the action $\Phi$ and
$$g (h x)= (g h) (g x)= (T_g h) (g x)\ \text{for all}\ x\in X, g\in G\ \text{and}\ h\in H,$$
recalling that $G$ acts on $H= X$ by automorphisms. Observe that by the well-known Birkhoff-Kakutani Theorem the compact metric group $X$ admits a left-invariant compatible metric (for example see \cite[Page 430, Lemma C.2]{EW}), which implies that the action $\Phi$ is equicontinuous. Thus the measure $\mu$ is $\alpha$-homogeneous by Lemma \ref{L-homogeneous}.
 \end{proof}

For a $\Zb$-action acting on a compact metric space $X$, if the action admits a homogeneous measure and has finite topological entropy, then it is asymptotically $h$-expansive \cite[Theorem 7.1]{Mi1}. In the following, we shall see that the proof there also works for a general countable discrete infinite amenable group $G$.

Let $(X, G)$ be an action of a countable group $G$ on a compact metric space $X$ with $\mu\in M (X), F\in\cF_G, \cU\in\cC^o_X$. We put
 $$M_F(\cU)=\max_{U\in \cU_F} \mu(U)> 0\ \text{and}\ m_F(\cU)=\min_{U\in P(\cU_F)} \mu(U).$$
For $\cU, \cV\in \cC^o_X$, $M_F(\cV)\leq M_F(\cU)$ whenever $\cV\succeq \cU$, and if $\mu$ is $\alpha$-homogeneous then
\begin{equation}
\label{E-2}
M_F(\cV)\le M_F (D (\cU))\leq c(\cU) m_F(\cU)\ \text{and hence}\ m_F(\cU)\ge \frac{M_F(\cV)}{c(\cU)}> 0
\end{equation}
whenever $\cV\succeq D(\cU)$.
Let $\cU\in\cC^o_X$ and $Y\subset X$. The \textit{star of Y with respect to $\cU$} and the {\it star of $\cU$} are defined, respectively, as
$$st(Y,\cU)=\cup\{U\in\cU: U\cap Y\neq \emptyset\}\ \text{and}\ St\cU=\{st(U,\cU):U\in\cU\}\in\cC^o_X.$$

Following \cite[Lemma 7.1]{Mi1}, it is easy to obtain:

\begin{lemma}
\label{L-1}
Let $(X, G)$ be an action of a countable group $G$ on a compact metric space $X$ with $\mu\in M (X)$ and $\cU,\cV\in \cC^o_X$ with $\cV\succeq\cU$. Then
$$\max_{K\in \cU_F} N((St\cV)_F, K)\cdot m_F (\cV)\leq M_F (St\cU).$$
\end{lemma}
\begin{proof}
Observe $St(\cV_F)\succeq (St\cV)_F$ as done in \cite[Lemma 3.3]{Mi1}. Let $U\in \cU_F$ with
\begin{equation} \label{1307112152}
\max_{K\in \cU_F} N ((St \cV)_F, K)\le \max_{K\in \cU_F} N (St (\cV_F), K)= N(St(\cV_F), U)=:p.
\end{equation}
Using \cite[Lemma 3.1]{Mi1}, we could select a disjoint family $\cC\subset P(\cV_F)$ such that $C\cap U\neq\emptyset$ for each $C\in\cC$ and $U\subset \cup \{st(C, P(\cV_F)): C\in\cC\}$. Observing $P(\cV_F)\subset \cV_F$ by the definition, $\cC\subset \cV_F$ (and hence  $st(C,\cV_F)\in St (\cV_F)$ for each $C\in \cC$) and $U\subset \cup \{st(C,\cV_F): C\in\cC\}$, which implies $p\leq |\cC|$.
 Now say $U=\bigcap_{\gamma\in F}\gamma^{-1}U_\gamma$ with $U_\gamma\in \cU$ for each $\gamma\in F$, then
 \begin{eqnarray} \label{1307112220}
st(U,\cV_F)&= & \cup \left\{W\in \cV_F: W\cap \bigcap_{\gamma\in F}\gamma^{-1}U_\gamma\neq \emptyset\right\}\nonumber \\
&\subset & \bigcap_{\gamma\in F}st(\gamma^{-1}U_\gamma,\gamma^{-1}\cV)= \bigcap_{\gamma\in F}\gamma^{-1}st(U_\gamma,\cV)\nonumber \\
&\subset & \bigcap_{\gamma\in F}\gamma^{-1}st(U_\gamma,\cU)\ (\text{as}\ \cV\succeq\cU)\ \in (St\cU)_F\ (\text{as $U_\gamma\in \cU$ for all $\gamma\in F$}).
 \end{eqnarray}
As $\cC\subset \cV_F$ and $C\cap U\neq\emptyset$ for each $C\in\cC$,
 $\cup \cC\subset st(U,\cV_F)$ by the definition. Now recalling that $\cC\subset P (\cV_F)$ is a disjoint family and $p\leq |\cC|$, we have $$p\cdot m_F(\cV)\leq \mu(\cup \cC)\leq \mu (st(U,\cV_F))\le M_F(St\cU)$$
 by \eqref{1307112220}. Then the conclusion follows from \eqref{1307112152} and the above estimation.
\end{proof}

The following result generalizes
\cite[Theorem 7.1]{Mi1} to the case of $G$ being a general countable discrete infinite amenable group.

\begin{proposition} \label{1307122314}
Let $\alpha$ be the action $(X, G)$ with $G$ a countable discrete infinite amenable group and $\mu\in M (X, G)$. Assume that $\mu$ is $\alpha$-homogeneous and $h^a (G, X)< \infty$. Then the action $(X, G)$ is asymptotically $h$-expansive.
\end{proposition}
\begin{proof}
As $\mu$ is $\alpha$-homogeneous, let $D:\cC^o_X\rightarrow \cC^o_X$ and $c: \cC^o_X\rightarrow (0,\infty)$ correspond to $\mu$.
Setting $\cU=\{X\}$ (and hence $St \cU= \{X\}$) in Lemma \ref{L-1} we obtain
\begin{equation}
\label{E-3}
N((St \cV)_F, X)\cdot m_F(\cV)\leq 1,\ \text{equivalently,}\ m_F(\cV)\le \frac{1}{N((St \cV)_F, X)}
\end{equation}
for each $\cV\in \cC^o_X$. And then once $\cV, \cW\in \cC^o_X$ satisfy $\cW\succeq D(\cV)$, $N(\cW_F, X)M_F(\cW)\geq 1$
and hence combining \eqref{E-2} with \eqref{E-3} we obtain
\begin{equation}
\label{E-5}
m_F(\cV)\geq \frac{M_F(\cW)}{c(\cV)}\geq \frac{1}{N(\cW_F, X)c(\cV)}\ \text{and}\
M_F(\cW)\leq c(\cV)m_F(\cV)\leq \frac{c(\cV)}{N((St\cV)_F, X)}.
\end{equation}

Now fix any $\cC, \cD\in \cC_X^o$. We choose $\cV\in \cC_X^o$ with $St \cV\succeq \cC, \cA\in \cC_X^o$ with $St\cA\succeq D(\cV)$ and $\cE\in \cC_X^o$ with $St \cE\succeq \cD$ by \cite[Proposition 3.5]{Mi1}; and then find $\mathcal{G}\in \cC_X^o$ with $\overline{\mathcal{G}}\succeq \cA$, $\cB\in \cC_X^o$ with $\cB\succeq \cA$ and $\cB\succeq \cE$ (and hence $St \cB\succeq \cD$) and $\cW\in \mathcal{C}_X^o$ with $\cW\succeq D (\cB)$. Choose $\cB'\in \cC^o_X$ such that $\cB'\succeq D(\cB)$, then by \eqref{E-2}, $m_F (\cB)\geq \frac{M_F(\cB')}{c(\cB)}> 0$ for any $F\in\cF_G$.
And hence we have
\begin{eqnarray*} \label{1307120037}
\max_{K\in (\overline{\mathcal{G}})_F} N((St\cB)_F, K)&\le & \max_{K\in \cA_F} N((St\cB)_F, K)\ (\text{as $\overline{\mathcal{G}}\succeq \cA$})\nonumber \\
&\le & \frac{M_F (St\cA)}{m_F (\cB)}\ (\text{using Lemma \ref{L-1}, as $\cB\succeq \cA$})\nonumber \\
&\le & c(\cV)\cdot c(\cB)\cdot \frac{N(\cW_F, X)}{N((St\cV)_F, X)}\ (\text{using \eqref{E-5}})
\end{eqnarray*}
as $St\cA\succeq D(\cV)$ and $\cW\succeq D (\cB)$, which implies 
$$h^a (G, St\cB| \overline{\mathcal{G}})\leq h^a(G,\cW)-h^a(G,St\cV)$$ from the definitions given at the end of section 5; and then (recalling that $G$ is a countable discrete infinite amenable group and observing $St \cV\succeq \cC$ and $St \cB\succeq \cD$)
\begin{equation*}
h^a (G, \cD| \overline{\mathcal{G}})\le h^a (G, St\cB| \overline{\mathcal{G}})\leq h^a(G,\cW)-h^a(G,St\cV)\le h^a (G, X)- h^a(G,\cC).
\end{equation*}
By the arbitrariness of $\cD\in \mathcal{C}_X^o$, one has
\begin{equation*}
h^{a, *} (G, X)\le h^a (G, X| \overline{\mathcal{G}})\le h^a (G, X)- h^a(G,\cC),
\end{equation*}
and then by the arbitrariness of $\cC\in \mathcal{C}_X^o$ and observing $h^a (G, X)< \infty$, we obtain $h^{a, *} (G, X)= 0$.
That is, $(X, G)$ is asymptotically $h$-expansive by Theorem \ref{1306160126}.
\end{proof}

Now we are ready to prove Theorem \ref{1307122130}.

\begin{proof}[Proof of Theorem \ref{1307122130}]
The ``only if'' part comes from Corollary \ref{1207122131}.

Now we prove the ``if'' part.
If $G$ is finite, it is trivial to see from the definition that $h^a (G, X)< \infty$ implies the finiteness of the group $X$, then the action $(X, G)$ is expansive and hence asymptotically $h$-expansive.
Now assume that $G$ is infinite.
Let $\mu\in M (X)$ be the normalized Haar measure of the compact metric group $X$.
  By Proposition \ref{1307122312} and Proposition \ref{1307122314}, we only need to show $\mu\in M (X, G)$, which is obvious by the assumption, as it is well known that each continuous surjective homomorphism of the compact metric group $X$ preserves $\mu$.
\end{proof}

It is natural to ask the following question which we haven't solved currently.

\begin{question}
Does Theorem \ref{1307122130} hold when the acting group is sofic?
\end{question}

At the end of this section, we present further discussions about actions of countable groups over a compact metrizable abelian group by continuous automorphisms.

Let $\alpha$ be an action of a countable group $G$ acting on a compact metrizable abelian group $X$ by continuous automorphisms.
On one hand, by \cite[Theorem 3.1]{ChungLi}, the action $(X, G)$ is expansive if and only if
there exist some $k\in \Nb$, some left $\Zb G$-submodule $J$ of $(\Zb G)^k$, and some  $A\in M_k(\Zb G)$ being invertible in $M_k (\ell^1 (G))$ such that
the left $\Zb G$-module $\widehat{X}$ is isomorphic to $(\Zb G)^k/J$ and the rows of $A$ are contained in $J$.
On the other hand, suppose that $G$ is amenable and $\widehat{X}$ is a finitely presented left $\Zb G$-module, and write $\widehat{X}$ as $(\Zb G)^k/(\Zb G)^n A$ for
some $k, n\in \Nb$ and $A\in M_{n\times k} (\Zb G )$, then $h^a (G, X)< \infty$ if and only if the additive map $(\Zb G)^k\rightarrow (\Zb G)^n$ sending $a$ to $a A^*$ is injective \cite[Theorem 4.11]{ChungLi}, which implies that $(X, G)$ is asymptotically $h$-expansive by Theorem \ref{1307122130}.
As a consequence, if $G$ is amenable, then for any non-zero divisor $f$ in $\Zb G$ which is not invertible in $\ell^1(G)$, the principal algebraic action $\alpha_f$, the canonical action of $G$ over $\widehat{\Zb G/\Zb G f}$, is asymptotically $h$-expansive but not expansive. See \cite[\S 3 and \S 4]{ChungLi} for more details.

\section*{Acknowledgements}

N-P. Chung would like to thank Lewis Bowen for asking and discussing the question ``When the actions of sofic groups are expansive, do they admit measures of maximal entropy?" which started this project. The joint work began when N.-P. Chung visited School of Mathematical Sciences, Fudan University on July, 2012. Part of the work was carried out while N.-P. Chung attended Arbeitsgemeinschaft ``Limits of Structures" at Oberwolfach on April, 2013. Most of this work was done while he was doing postdoc at Max Planck Institute, Mathematics in the Sciences, Leipzig.

Both authors thank Hanfeng Li for helpful suggestions which resulted in substantial improvements to this paper. The authors are also grateful to Tomasz Downarowicz, Ben Hayes and Wen Huang for helpful comments, and to Thomas Ward for bringing the reference \cite{Berg} into our attention which is related to Remark \ref{1409270036}.

This research was supported through the programme ``Research in Pairs" by the Mathematisches Forschungsinstitut Oberwolfach in 2013. We are grateful to MFO for a warm hospitality. N.-P. Chung was also supported by Max Planck Society, and
G. H. Zhang is also supported by FANEDD (201018) and NSFC (11271078).


\bibliographystyle{amsalpha}

\end{document}